\documentclass[journal,twocolumn,10pt]{IEEEtran}
\usepackage{setspace}
\usepackage{amsmath,amssymb,epsfig,amsthm}
\usepackage[ruled,lined,linesnumbered]{algorithm2e}
\usepackage{algorithmic,multirow,array,graphicx,url,isolatin1}

\newtheorem{prop}{Proposition}

\DeclareMathOperator*{\prox}{prox}
\newcommand{\bzeta}{\mbox{\boldmath $\zeta$}}
\newcommand{\bone}{\mbox{\boldmath $1$}}
\newcommand{\bzero}{\mbox{\boldmath $0$}}

\newcommand{\bUpsilon}{\mbox{\boldmath $\Upsilon$}}
\newcommand{\bDelta}{\mbox{\boldmath $\Delta$}}
\newcommand{\bLambda}{\mbox{\boldmath $\Lambda$}}
\newcommand{\bdelta}{\mbox{\boldmath $\delta$}}

\title{\huge Deconvolving Images with Unknown Boundaries \\ Using the Alternating  Direction Method of Multipliers}
\author{Mariana S. C. Almeida \ \  {\rm and} \ \   M\'{a}rio A. T. Figueiredo, \textit{Fellow, IEEE}\thanks{
Both authors are with the Ins\-ti\-tu\-to
de Telecomunica\c{c}\~{o}es (IT), Instituto Superior T\'{e}cnico (IST), 1049-001 Lisboa, Portugal.
Email: {\tt $\{$mariana.almeida,mario.figueiredo$\}$@lx.it.pt}

This work was partially supported by {\it Funda\c{c}\~{a}o para a Ci\^{e}ncia e Tecnologia} (FCT), under grants PTDC\-/EEA-TEL\-/104515\-/2008 and
PEst-OE/EEI/LA0008/2011, and the fellowship SFRH\-/BPD/\-69344/2010.}}

\begin{document}

%\ninept
%\onehalfspacing
\maketitle

\begin{abstract}
The \emph{alternating direction method of multipliers} (ADMM) has recently
sparked interest as a flexible and efficient optimization tool for imaging inverse
problems, namely deconvolution and reconstruction under non-smooth convex
regularization. ADMM achieves state-of-the-art speed by
adopting a \emph{divide and conquer}
strategy, wherein a hard problem is split into simpler, efficiently solvable
sub-problems ({\it e.g.}, using fast Fourier or wavelet transforms, or
simple proximity operators). In deconvolution, one of these
sub-problems involves a matrix inversion ({\it i.e.}, solving a linear system),
which can be done efficiently (in the discrete Fourier domain) if
the observation operator is circulant, {\it i.e.}, under periodic boundary
conditions. This paper extends ADMM-based image deconvolution to the more
realistic scenario of unknown boundary, where the observation operator
is modeled as the composition of a convolution (with arbitrary boundary
conditions) with a spatial mask that keeps only pixels that do not depend
on the unknown boundary. The proposed approach also handles,
at no extra cost, problems that combine the recovery of missing
pixels ({\it i.e.}, inpainting) with deconvolution.
We show that the resulting algorithms inherit
the convergence guarantees of ADMM and illustrate its performance on
non-periodic deblurring (with and without inpainting of interior pixels)
under total-variation and frame-based regularization.

\end{abstract}


\section{Introduction}
\label{sec:intro}

%ADMM inicio 1
%ADMM recentemente 2
%Rela\c{c}ao com outros 1
%boyd 1
%Denoising 1
%Nonblind 1 ok
%BID 3 (TIP, TV, ?)
%Boundary 1-3 (padding, reflective (anti reflective), edgetaper, Jia, Manya,
%\renewcommand{\baselinestretch}{2}

The \emph{alternating direction method of multipliers} (ADMM), originally
proposed in the 1970's \cite{Gabay_ADMM_76}, emerged recently as a
flexible and efficient tool for several imaging inverse problems, such as denoising
\cite{FigBioucas_10,Steidl_Denoising_10}, deblurring \cite{AfonsoFig_TIP10}, inpainting \cite{AfonsoFig_TIP10},
reconstruction \cite{Ramani}, motion segmentation \cite{Zappella}, to mention only a few classical
problems (for a comprehensive review, see \cite{Boyd_ADMM_11}). ADMM-based approaches make use of variable splitting, which allows a straightforward
treatment of various priors/regularizers \cite{AfonsoICIP2010}, such as those based on
frames or on {\it total-variation} (TV) \cite{ROF}, as well as the seamless
inclusion of several types of constraints ({\it e.g.}, positivity)
\cite{FigBioucas_10,Steidl_Denoising_10}. ADMM is closely related to other
techniques, namely the so-called {\it Bregman} and {\it split Bregman} methods
\cite{Cai09,Goldstein_Stanley_SBL1_2009,YinOsher2008} and  {\it Douglas-Rachford
splitting} \cite{Boyd_ADMM_11,CombettesPesquet2007,Eckstein92,Esser09}.

Several ADMM-based algorithms for imaging inverse problems require, at each iteration,
solving a linear system (equivalently, inverting a
matrix) \cite{AfonsoFig_TIP10,Cai09,FigBioucas_10,Pustelnik2012,Tao}.
This fact is simultaneously a blessing and a curse. On the one hand, the matrix to
be inverted is related to the Hessian of the objective function, thus
carrying second order information; arguably, this fact
justifies the excellent speed of these methods, which have
been shown (see, {\it e.g.}, \cite{AfonsoFig_TIP10}) to be considerably faster
than the classical {\it iterative shrinkage-thresholding}  (IST)
algorithms \cite{Daubechies04,FigueiredoNowak2003} and even than
their accelerated versions \cite{FISTA,TwIST,SpaRSA}.
On the other hand, this inversion (due to its typically huge size)
limits its applicability to problems where it can be efficiently computed
(by exploiting some particular structure).
For ADMM-based image deconvolution algorithms \cite{AfonsoFig_TIP10,Cai09,FigBioucas_10,Tao},
this inversion can be efficiently carried out using the {\it fast
Fourier transform} (FFT), if the convolution is cyclic/periodic
(or assumed to be so), thus diagonal in the discrete Fourier
domain. However, as explained next, periodicity is an unnatural
assumption, inadequate for most real imaging problems.

In deconvolution, the pixels located near the boundary of
the observed image depend on pixels (of the unknown image) located
outside of its domain. The typical way to formalize this
issue is to adopt a so-called
{\it boundary condition} (BC).
\begin{itemize}
\item
The {\bf periodic} BC (the use of which, in image deconvolution,
dates back to the 1970s \cite{AndrewsHunt}) assumes a periodic
convolution; its matrix representation is circulant\footnote{In fact,
when dealing with 2-dimensional (2D) images,
these matrices are block-circulant with circulant blocks,
thus diagonalized by the 2D {\it discrete Fourier transform} (DFT).
For simplicity, we refer to such matrices simply as
circulant and to the 2D DFT simply as DFT.},
\addtocounter{footnote}{-1} diagonalized by the
DFT\footnotemark, which can be implemented via the FFT.
This computational convenience makes it, arguably, the most commonly adopted BC.

\item
The {\bf zero} BC assumes that all the external pixels have zero
value, thus the matrix representing the convolution is
block-Toeplitz, with Toeplitz blocks \cite{Ng_book}. By analogy with the BC for ordinary or
partial differential equations that assumes fixed values at the domain
boundary, this is commonly referred to Dirichlet BC \cite{Ng_book}.

\item
In the {\bf reflexive} and {\bf anti-reflexive} BCs, the pixels outside the
image domain are a reflection of those near the boundary, using even or odd
symmetry, respectively \cite{Chan2005,Donatelli_BC_06,NgChanTang1999}.
In the reflexive BC, the discrete derivative at the boundary is zero; thus, by analogy
with the BC for ordinary or partial differential equations  that assumes
fixed values of the derivative at the boundary, the reflexive BC
is often referred to as Neumann BC \cite{Ng_book}.
\end{itemize}

\begin{figure}
		\centering
	 %\begin{tabular}{c @{ } c @{ } c} %@{  }c
	 \begin{tabular}{c   c   c} %@{  }c
  \includegraphics[width=0.31\columnwidth]{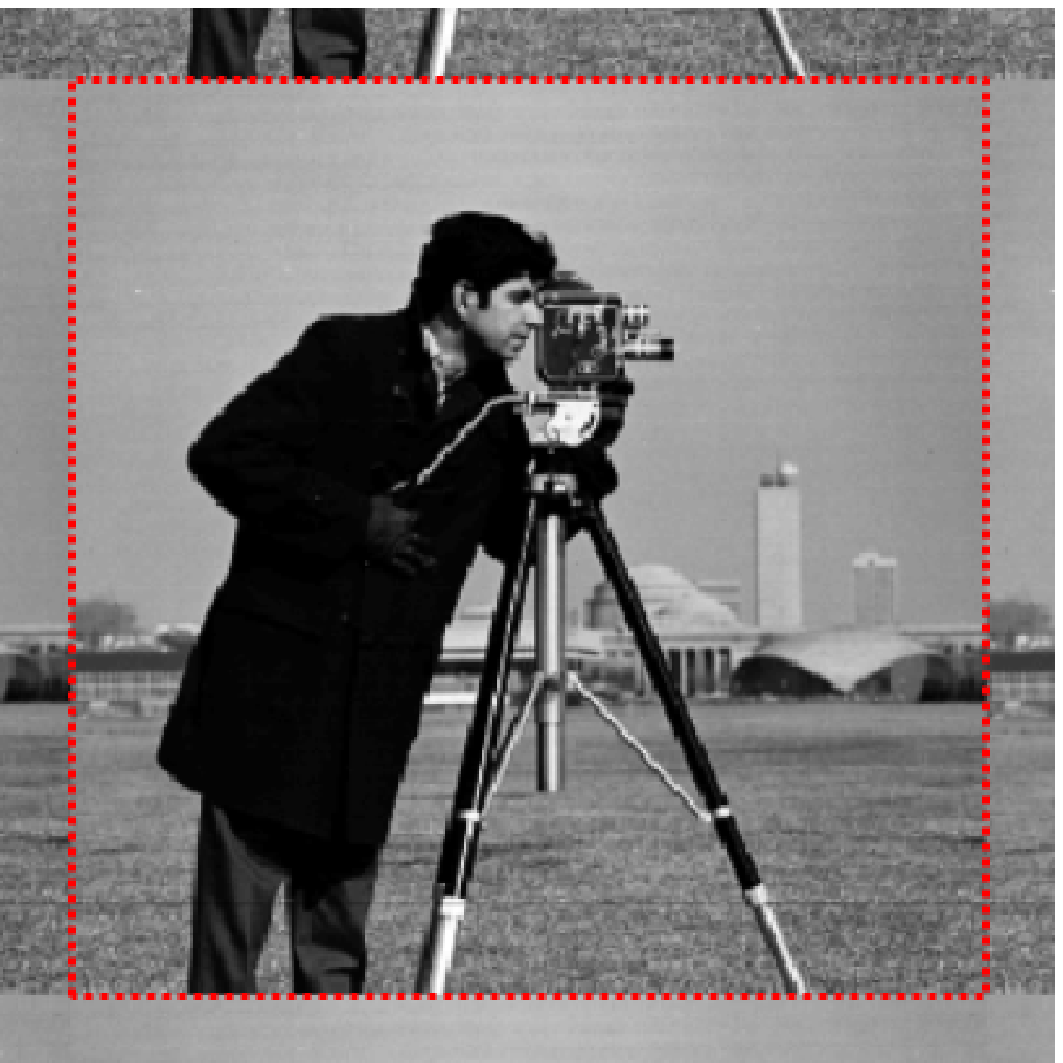}
   \includegraphics[width=0.31\columnwidth]{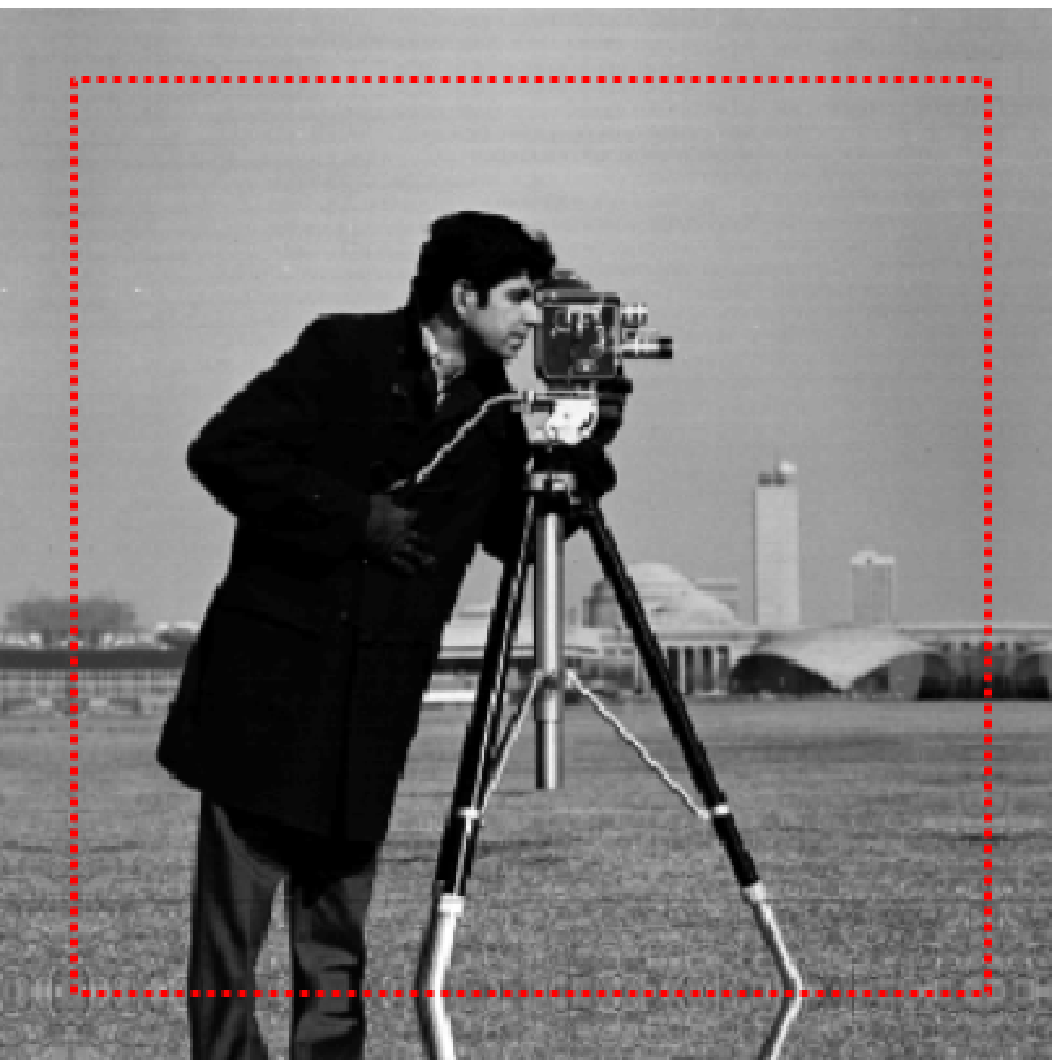}
   \includegraphics[width=0.31\columnwidth]{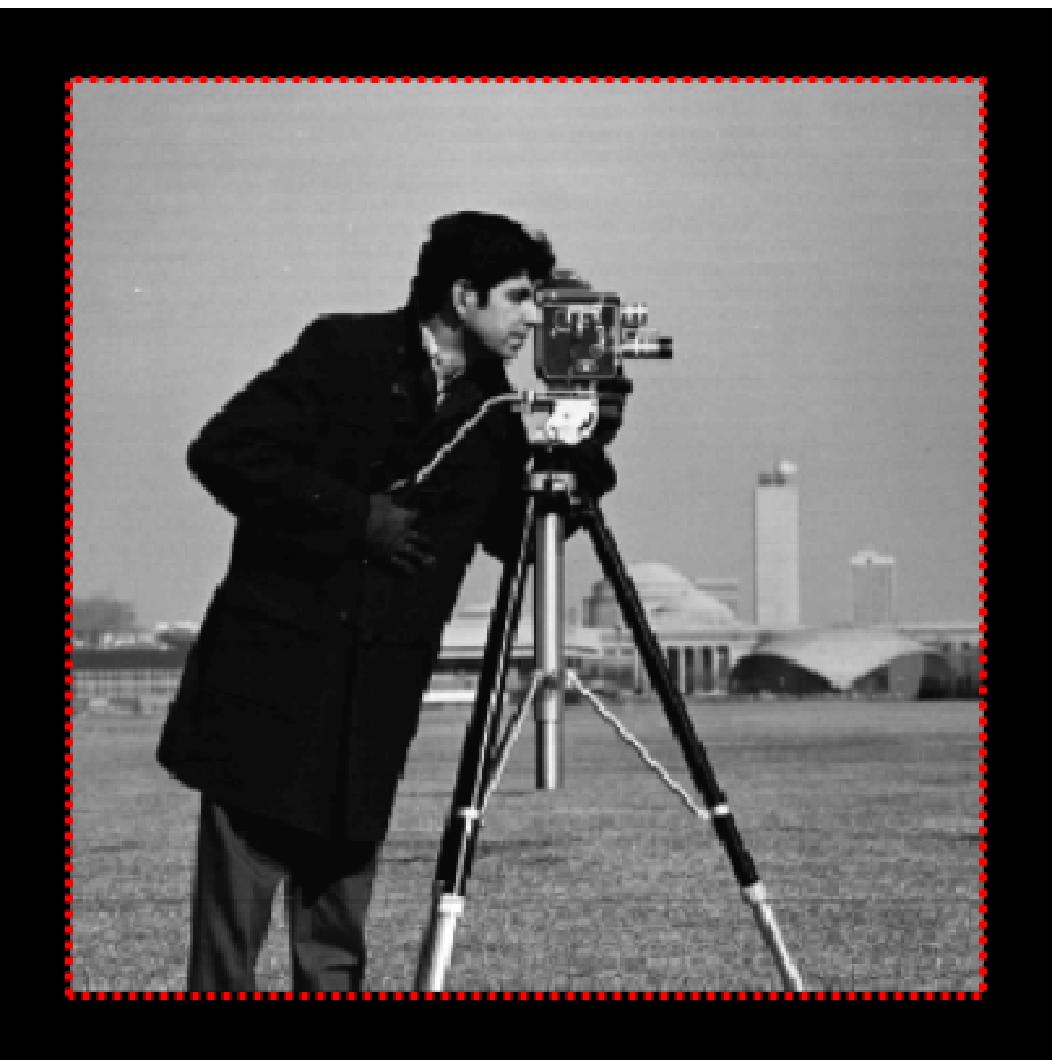}
    \end{tabular}
\caption{Illustration of the (unnatural) assumptions underlying the periodic, reflexive, and
zero boundary conditions.}
\label{fig:boundaries}
\end{figure}

As illustrated in Figure~\ref{fig:boundaries}, all these standard BCs
are quite unnatural, as they do not correspond to any realistic imaging
system, being motivated merely by computational convenience.
Namely, assuming a periodic boundary condition has the advantage of allowing a very fast
implementation of the convolution as a point-wise multiplication in the
DFT domain, efficiently implementable using the FFT.
However, in real problems, there is no reason for the external (unobserved)
pixels to follow periodic (or any other) boundary conditions. A well known consequence
of this mismatch is a degradation of the deconvolved images, such as the appearance
of ringing artifacts emanating from the boundaries
(see Figs. \ref{fig:Lena_uniform} and \ref{fig:Cameraman_motion}).
These artifacts can be reduced by pre-processing the image to reduce
the spurious discontinuities at the boundaries, created by the (wrong)
periodicity assumption; this is what is done by the ``edgetaper" function in
the {\it MATLAB Image Processing Toolbox}. In a more sophisticated version
of this idea that was recently proposed, the observed image is extrapolated
to create a larger image with smooth BCs \cite{LiuJia_BC_08}.

Convolutions under zero BC can still be efficiently performed in the DFT domain,
by embedding (using zero-padding) the non-periodic convolution
into a larger periodic one \cite{Ng_book}. However,
in addition to the mismatch problem pointed out in the previous paragraph
({\it i.e.}, there is usually no reason to assume that the boundary is zero),
this choice has another drawback: in the context of ADMM-based methods,
the zero-padding technique prevents the required matrix inversion
from being efficiently computed via the FFT \cite{AfonsoFig_NonCyclic_11}.

\subsection{Related Work and Contributions}
\label{ssec:contrib}
Instead of adopting a standard BC or a boundary
smoothing scheme, a  more realistic model of actual
imaging systems treats the external boundary pixels as unknown;
{\it i.e.}, the problem is seen as one of simultaneous
deconvolution and inpainting, where the unobserved boundary pixels
are estimated together with the deconvolved image.
The corresponding observation model is the composition of a spatial mask
with a convolution under arbitrary BC \cite{Chan2005,MatakosRamaniFessler_2012,Reeves05,Sorel12};
addressing this formulation using ADMM is the central theme of this paper.
While we were concluding this manuscript, we became aware of very recent
related work in \cite{MatakosRamaniFessler_2013}.

Under quadratic (Tikhonov) regularization, image deconvolution with
periodic BC corresponds to a linear system, where the corresponding matrix
can be efficiently inverted in the DFT domain using the FFT.
In the unknown boundary case, it was shown in \cite{Reeves05}
how this inversion can be reduced to an FFT-based inversion followed
by the solution (via, {\it e.g.}, {\it conjugate gradient} -- CG) of
a much smaller system (same dimension as the unknown boundary).
Very recently, \cite{Sorel12} adapted this technique to
deconvolution under TV and frame-based analysis non-smooth
regularization; that work proposes an algorithm based
on variable splitting and quadratic penalization, using the method of \cite{Reeves05}
to solve the linear system at each iteration. That method is related
to, but it is not ADMM, thus has no guarantees of converge to a minimizer
of the original objective function. Although \cite{Sorel12} mentions the
possibility of using the method of \cite{Reeves05} within ADMM, that option
was not explored.

In this paper, we address image deconvolution with unknown boundaries,
under TV-based and frame-based (synthesis and analysis) non-smooth
regularization, using  ADMM. Our first and main approach
exploits the fact that the observation model involves the composition of
a spatial mask with a periodic convolution, and uses ADMM in a way that
decouples these two operations. The resulting algorithms inherit
all the desirable properties of previous ADMM-based deconvolution
methods: all the update equations (including
the matrix inversion) can be computed efficiently without using inner
iterations; convergence is formally guaranteed. The second approach considered
is the direct extension of the methods from
\cite{AfonsoFig_TIP10,Cai09,FigBioucas_10,Tao} to the unknown boundary
case; here (unlike with periodic BC), the linear system at each iteration
cannot be solved efficiently using the FFT, thus we adopt
the technique of \cite{Reeves05,Sorel12}. However, since the
resulting algorithms are instances of ADMM, they have
convergence guarantees, which are missing in \cite{Sorel12}.
Furthermore, we show how the method of \cite{Reeves05} can
also be used in ADMM for frame-based synthesis regularization,
whereas \cite{Sorel12} only considers analysis formulations.

The proposed methods are experimentally illustrated using frame-based
and TV-based regularization; the results show the
advantage of our approach over the use of the ``edgetaper'' function (in terms
of improvement in SNR) and over an adaptation of the recent strategy of \cite{Sorel12}
(in terms of speed). Finally, our approach is also tested on inverse problems where the
observation model consists of a (non-cyclic) deblurring followed by a generic loss of
pixels (inpainting problems). Arguably due to its complexity, this problems has not
been often addressed in the literature, although it is a relevant one:
consider the problem of deblurring an image in which some
pixel values are unaccessible, {\it e.g.}, because they are saturated, thus should
not be taken into account, or because they correspond to so-called {\it dead pixels} in
the image sensor.

\subsection{Outline and Notation}
\label{ssec:outline}

Sections \ref{sec:ADMM} and \ref{sec:TV_Deblurring} review the ADMM
and its use for image deconvolution with periodic BC, using frame-based
and TV-based regularization,  setting the stage for the proposed
approaches for the unknown boundary case, which are introduced in
Section \ref{sec:proposed}. The experimental
evaluation of our methods is reported in Section \ref{sec:Exp},
which also illustrates its use for simultaneous deblurring
and inpaiting. Section \ref{sec:Conc} concludes the manuscript.

We use the following notation: $\bar{\mathbb{R}} = \mathbb{R} \cup \{-\infty,+\infty\}$
is the extended real line; bold lower case denotes vectors ({\it e.g.}, ${\bf x}$, ${\bf y}$),
and bold upper case (Roman or Greek) are matrices ({\it e.g.}, ${\bf A}$, $\bUpsilon$);
the superscript  $(\cdot)^*$ denotes vector/matrix transpose, or conjugate transpose in the
complex case;  ${\bf I}$,  $\bone $, and $\bzero$ are the identity matrix  and
vectors with all elements equal to one and zero, respectively;
as usual, $\circ$ denotes function composition  ({\it e.g.}, $(f\circ{\bf A})({\bf x}) = f({\bf Ax})$),
and $\odot$ the component-wise product between vectors ($({\bf v}\odot{\bf u})_i = v_i \, u_i$).

\section{Alternating Direction Method of Multipliers}
\label{sec:ADMM}

\subsection{The Standard ADMM Algorithm}
\label{ssec:ADMM1}

%\subsection{ADMM}
%\label{ssec:ADMM}

Consider an unconstrained optimization problem
\begin{equation}
\min_{{\bf z} \in \mathbb{R}^n}  f(\textbf{z}) + g(\textbf{G}\textbf{z}),
\label{eq:unconstrained}
\end{equation}
where $f:\mathbb{R}^n \rightarrow \bar{\mathbb{R}}$ and $g:\mathbb{R}^p
\rightarrow \bar{\mathbb{R}}$ are convex functions, and
${\bf G}\in\mathbb{R}^{p\times n}$ is a matrix. The  ADMM algorithm
for this problem (which can be derived from different perspectives,
namely, augmented Lagrangian and Douglas-Rachford splitting)  is
presented in Algorithm \ref{alg:ADMM}; for a recent comprehensive
review, see \cite{Boyd_ADMM_11}. Convergence of ADMM was shown
in \cite{Eckstein92}, where the following theorem was proved.

\newtheorem{theorem}{Theorem}

\begin{theorem}
\label{th:Eckstein}{ Consider problem  (\ref{eq:unconstrained}), where
 ${\bf G}\in\mathbb{R}^{p\times n}$ has full column rank and
 $f\,:\,\mathbb{R}^d\rightarrow \bar{\mathbb{R}}$ and $g\,:\,\mathbb{R}^p\rightarrow
\bar{\mathbb{R}}$ are closed, proper, convex functions. Consider arbitrary  $\mu>0$,
${\bf u}_0, {\bf d}_0\in \mathbb{R}^p$.
Let $\eta_k \geq 0$ and $\rho_k \geq 0,$ for $k=0,1,...,$ be
two sequences such that $\sum_{k=0}^\infty \eta_k < \infty$
and $\sum_{k=0}^\infty \rho_k < \infty.$
Consider three sequences ${\bf z}_k \in \mathbb{R}^{n}, {\bf u}_k
\in \mathbb{R}^{p}$, and ${\bf d}_k \in \mathbb{R}^{p}$, for $k=0,1,...,$ satisfying
\begin{eqnarray}
  \left\| {\bf z}_{k+1} - \arg\min_{{\bf z}} f ({\bf z})
 + \frac{\mu}{2} \|{\bf G}{\bf z} \! - \!{\bf u}_k \! -\! {\bf d}_k\|_2^2 \right\| & \leq & \eta_k \nonumber\\
  \left\| {\bf u}_{k+1}  - \arg\min_{{\bf u}} g({\bf u})
 + \frac{\mu}{2} \|{\bf G}{\bf z}_{k+1} \! - \! {\bf u} \! - \! {\bf d}_k\|_2^2 \right\| & \leq &  \rho_k \nonumber\\
 {\bf d}_{k+1}  =  {\bf d}_{k} - ({\bf G\, z}_{k+1} - {\bf u}_{k+1}).\nonumber
\end{eqnarray}
Then, if (\ref{eq:unconstrained}) has a solution, say ${\bf z}^*$,
the sequence $\{{\bf z}_k\}$  converges to ${\bf z}^*$.
If (\ref{eq:unconstrained}) does not have a solution,  at least
one of the sequences $({\bf u}_1,  {\bf u}_2, ...)$ or $({\bf d}_1, {\bf d}_2, ...)$ diverges.}
\end{theorem}

\begin{algorithm}
\DontPrintSemicolon
\caption{ADMM\label{alg:ADMM}}
Initialization: set $k=0$, choose $\mu > 0$, $\textbf{u}_0$, and $\rule[-0.25cm]{0cm}{0.6cm}\textbf{d}_0$.\;
\Repeat{stopping criterion is satisfied}{
$\rule[-0.2cm]{0cm}{0.4cm}\textbf{z}_{k+1} \leftarrow \mathop{\arg\min}_{\textbf{z}} f(\textbf{z})
														+ \frac{\mu}{2} || \textbf{ G z} - \textbf{u}_{k} - \textbf{d}_k || _2^2 $ \\
$\rule[-0.2cm]{0cm}{0.4cm}\textbf{u}_{k+1} \leftarrow \mathop{\arg\min}_{\textbf{u}} g(\textbf{u})
														+ \frac{\mu}{2} ||\textbf{ G  z}_{k+1}  - \textbf{u} - \textbf{d}_k || _2^2 $ \\
$\rule[-0.2cm]{0cm}{0.4cm} \textbf{d}_{k+1} \leftarrow \textbf{d}_k - ( \textbf{Gz}_{k+1} - \textbf{u}_{k+1}$ ) \\
$k \leftarrow k+1$
}
\end{algorithm}

There are more recent results on ADMM, namely establishing linear convergence \cite{DengYin2012};
however, those results make stronger assumptions (such as strong convexity) which are not generally
satisfied in deconvolution problems. Under the conditions of \textit{Theorem 1},
convergence holds for any $\mu>0$; however, the choice of
$\mu$ may strongly affect the convergence speed \cite{Boyd_ADMM_11}.
It is also possible to replace the scalar $\mu$ by a positive diagonal matrix $\bUpsilon
= \mbox{diag}(\mu_1,...,\mu_p)$,
{\it i.e.}, to replace the quadratic terms of the form $\mu \|{\bf G}{\bf z} - {\bf u}- {\bf d}\|_2^2$
by $({\bf G}{\bf z} - {\bf u}- {\bf d})^* \bUpsilon ({\bf G}{\bf z} - {\bf u}- {\bf d})$.

\subsection{The ADMM For Two or More Functions}
\label{ssec:ADMM2}
Following the formulation proposed in \cite{Afonso2011,FigBioucas_10}, consider a generalization
of \eqref{eq:unconstrained}, with $J \geq 2$ functions,
\begin{equation}
\mathop{\min}_{\textbf{z} \in \mathbb{R}^n}  \sum_{j=1}^{J} g_j(\textbf{H}^{(j)}\textbf{z}),
\label{eq:unconstrained2}
\end{equation}
where $g_j : \mathbb{R}^{p_j}\rightarrow \bar{\mathbb{R}}$ are proper, closed, convex functions,
and $\textbf{H}^{(j)} \in \mathbb{R}^{p_j\times n}$ are arbitrary matrices. We map this
problem into form \eqref{eq:unconstrained} as follows: let $f(\textbf{z})=0$; define matrix ${\bf G}$ as
\begin{equation}
\textbf{G} =  \begin{bmatrix} \textbf{H}^{(1)} \\ \vspace{-0.5cm} \\ \vdots \\ \ \vspace{-0.4cm}\\ \textbf{H}^{(J)}\end{bmatrix}  \in  \mathbb{R}^{p\times n},
\label{eq:G}
\end{equation}
where $p = p_1+...+p_{J}$, and let $g: \mathbb{R}^{p} \rightarrow \bar{\mathbb{R}}$ be defined as
\begin{equation}
g(\textbf{u}) = \sum_{j=1}^{J} g_j(\textbf{u}^{(j)}),
\label{eq:G2}
\end{equation}
where each $\textbf{u}^{(j)}  \in \mathbb{R}^{p_j}$ is\  a\  $p_j$-dimensional
sub-vector of ${\bf u}$, that is,  $\textbf{u} = \left[ (\textbf{u}^{(1)})^*\! , ...,  (\textbf{u}^{(J)})^* \right]^*$.

The definitions in the previous paragraph lead to an ADMM for problem \eqref{eq:unconstrained2}
with the following two key features.
\begin{enumerate}
\item The fact that $f({\bf z}) = 0$ turns line 3 of Algorithm \ref{alg:ADMM}
into an unconstrained quadratic problem. Letting $\bUpsilon$ be a $p$-dimensional
positive block diagonal matrix (associating a possibly different parameter $\mu_j$ to each function  $g_j$)
\begin{equation}
\bUpsilon = \mbox{diag}\bigl(\!\! \underbrace{\mu_1 ,..., \mu_1}_{p_1 \; \mbox{elements}}\!\! , ...,
\!\!\underbrace{\mu_j ,..., \mu_j}_{p_j \; \mbox{elements}}\!\!, ...,\!\!\underbrace{\mu_J, ...,\mu_J}_{p_J \; \mbox{elements}}\! \bigr),
\label{eq:diagonal_penalty}
\end{equation}
the solution of this quadratic problem is given by (denoting $\bzeta = {\bf u}_k +{\bf d}_k$)
\begin{multline}
\arg\min_{{\bf z}}  \bigl( {\bf G\, z} - \bzeta\bigr)^* \bUpsilon  \bigl( {\bf G\, z} - \bzeta\bigr)
= \Bigl( {\bf G}^* \bUpsilon {\bf G}\Bigr)^{-1} \! {\bf G}^* \bUpsilon \bzeta  \\
 =  \biggl[ \sum_{j=1}^J \mu_j \,
({\bf H}^{(j)})^*{\bf H}^{(j)}\biggr]^{-1} \sum_{j=1}^J \, \mu_j \, \bigl({\bf H}^{(j)}\bigr)^* \bzeta^{(j)} ,\label{eq:quadratic_problem}
\end{multline}
with $\bzeta^{(j)} = {\bf u}_k^{(j)} + {\bf d}_k^{(j)}$, where $\bzeta^{(j)}$,
${\bf u}_k^{(j)}$, and ${\bf d}_k^{(j)}$, for $j=1,...,J$, are the sub-vectors of
$\bzeta$, ${\bf u}_k$, and ${\bf d}_k$, respectively, corresponding to the
partition in \eqref{eq:G}, and the second equality results from the block structure
of matrices ${\bf G}$ (in \eqref{eq:G}) and $\bUpsilon$ (in \eqref{eq:diagonal_penalty}).

\item The separable structure of $g$ (as defined in \eqref{eq:G2}) allows decoupling
the minimization in line 4 of Algorithm \ref{alg:ADMM} into $J$ independent minimizations,
each of the form
\begin{equation}
{\bf u}_{k+1}^{(j)} \leftarrow \arg\min_{{\bf v}\in \mathbb{R}^{p_j}} \;  g_j ({\bf v}) +
\frac{\mu_j}{2} \, \bigl\| {\bf v}  - {\bf s}^{(j)}\bigr\|_2^2 , \label{eq:Moreau_1}
\end{equation}
for $j=1,...,J$, where ${\bf s}^{(j)} =  {\bf H}^{(j)} {\bf z}_{k+1} - {\bf d}^{(j)}_k.$
This minimization defines to the so-called {\it Moreau proximity operator}
of $g_j/\mu$ (denoted as $\mbox{prox}_{g_j/\mu_j}$) (see \cite{CombettesPesquet2007,CombettesSIAM}
and references therein)
applied to ${\bf s}^{(j)}$, thus
\[
{\bf u}_{k+1}^{(j)} \leftarrow \mbox{prox}_{g_j/\mu_j}( {\bf s}^{(j)} ) \equiv
\arg\min_{{\bf x}} \frac{1}{2} \bigl\| {\bf x} - {\bf s}^{(j)} \bigr\|_2^2 + \frac{g_j ({\bf x})}{\mu_j}.
\]
For some functions, the Moreau proximity operators
are known in closed form \cite{CombettesPesquet2007}; a well-known case is
the $\ell_1$ norm ($\tau \|{\bf x}\|_1 = \tau \sum_i |x_i|$), for which
the proximity operator is the soft-threshold function:
\begin{equation}
\mbox{soft}(v,\gamma) = \mbox{sign}({\bf v}) \odot \max\{ |{\bf v}| - \tau ,0\},\label{eq:soft}
\end{equation}
where the sign, max, and absolute value functions are applied in component-wise fashion.
\end{enumerate}

The convergence of the resulting instance of ADMM (shown
in Algorithm \ref{alg:ADMM2}) is the subject of the following
proposition.

\begin{prop}
\label{th:Eckstein2}
Consider problem \eqref{eq:unconstrained2}, where $g_j : \mathbb{R}^{p_j}\rightarrow
\bar{\mathbb{R}}$ are proper, closed, convex functions, and
$\textbf{H}^{(j)} \in \mathbb{R}^{p_j\times n}$. Consider positive
constants $\mu_1,...,\mu_J > 0$ and arbitrary ${\bf u}_0, {\bf d}_0\in \mathbb{R}^p$.
If matrix ${\bf G}$ has full column rank and (\ref{eq:unconstrained2}) has a solution, say ${\bf z}^*$,
then, the sequence $({\bf z}_1,{\bf z}_2,...)$  generated by Algorithm 2 converges to ${\bf z}^*$.
\end{prop}

\begin{proof}
We simply need to show that the conditions of \textit{Theorem 1} are
satisfied. Problem \eqref{eq:unconstrained2} has the form \eqref{eq:unconstrained},
with $f({\bf z})=0$ and $g$ as given by \eqref{eq:G2}; if all the functions $g_1,...,g_J$ are
closed, proper, and convex, so are $g$ and $f$. Furthermore, if ${\bf G}$ has full column rank,
a sufficient condition for the inverse in \eqref{eq:quadratic_problem} to
exist is that $\mu_1,...,\mu_J > 0$. Finally, the minimization in line 3
of Algorithm~\ref{alg:ADMM} is solved exactly in line 4 of
Algorithm~\ref{alg:ADMM2}, and the minimization in line 4 of Algorithm~\ref{alg:ADMM}
is solved exactly in line 6 of Algorithm~\ref{alg:ADMM2}. Thus, we can take
the sequences $\eta_k$ and $\rho_k$ in Theorem~\ref{th:Eckstein} to be identically
zero.
\end{proof}

\begin{algorithm}
\DontPrintSemicolon
\caption{ADMM-2 \label{alg:ADMM2}}
Initialization: set $k=0$, choose $\mu_1,...,\mu_J > 0$, $\textbf{u}_0$, $\rule[-0.1cm]{0cm}{0.2cm}\textbf{d}_0$.\;
\Repeat{stopping criterion is satisfied}{
$\rule[-0.18cm]{0cm}{0.36cm}\bzeta \leftarrow \textbf{u}_k +\textbf{ d}_k$\\
$\rule[-0.5cm]{0cm}{0.5cm}{\displaystyle \textbf{z}_{k+1} \leftarrow \Bigl(\sum_{j=1}^{J} \mu_j (\textbf{H}^{(j)})^* \textbf{H}^{(j)} \Bigr)^{-1}
		 \sum_{j=1}^{J} \mu_j (\textbf{H}^{(j)})^* \bzeta^{(j)}}$	\\							
\For{$j=1$ {\rm to} $J$}{$\rule[-0.3cm]{0cm}{0.6cm}\textbf{u}_{k+1}^{(j)} \leftarrow \prox_{g_j / \mu_j} (\textbf{H}^{(j)}\textbf{z}_{k+1} - \textbf{d}_k^{(j)}) $ \\
$\rule[-0.25cm]{0cm}{0.5cm}\textbf{d}_{k+1}^{(j)} \leftarrow \textbf{d}_k^{(j)} - ( \textbf{H}^{(j)}\textbf{z}_{k+1} - \textbf{u}_{k+1}^{(j)}$ )}
$\rule[-0.15cm]{0cm}{0.3cm}k \leftarrow k+1$}
\end{algorithm}

If the proximity operators of the functions $g_j$ are simple,
the computational bottleneck of Algorithm \ref{alg:ADMM2}  is the matrix inversion in line 4.
As shown in~\cite{Tao,WangZhang_ADMMTV_08}, inversions of this form can
be efficiently tackled using the FFT, if all the matrices ${\bf H}^{(j)}$ are
circulant, thus jointly diagonalized by the DFT (more details in Section~\ref{sec:TV_Deblurring}).
Due to this condition, the work in \cite{Tao,WangZhang_ADMMTV_08} was limited to
deconvolution with periodic BCs, under TV regularization.
More recent work in \cite{AfonsoFig_TIP10,Afonso2011,FigBioucas_10} has
shown how these inversions can still be efficiently handled via the FFT
(and other fast transforms) in problems such as image reconstruction from
partial Fourier observations and inpainting, and with other regularizers,
such as those based on tight frames. This paper extends that work,
showing how to handle deconvolution problems with unknown boundaries.

\section{Image Deconvolution with Periodic BC}
\label{sec:TV_Deblurring}
This section reviews the ADMM-based approach to image deconvolution with
periodic BC, using the frame-based formulations and TV-based
regularization \cite{AfonsoFig_TIP10,Afonso2011,FigBioucas_10,ROF,Tao},
the standard regularizers for this class of imaging inverse problems.
The next section will then show how this approach can be extended to
deal with unknown boundaries.

\subsection{Frame-Based Synthesis Formulation}
\label{sec:synth_periodic}
There are two standard ways to formulate frame-based regularization for
image deconvolution, both exploiting the fact that natural images
have sparse representations on wavelet\footnote{We use the generic term ``wavelet" to mean any
wavelet-like multiscale representation, such as ``curvelets," ``beamlets," or ``ridgelets" \cite{Mallat}.} frames:
the synthesis and analysis formulations
\cite{AfonsoFig_TIP10,EladMilanfar2007,SelesnickFigueiredo2009}.

The synthesis formulation expresses the estimated image $\widehat{\bf x}$ as a
linear combination of the elements of some wavelet
frame (an orthogonal basis or an overcomplete dictionary), {\it i.e.}, $\widehat{\textbf{x}} = {\bf W}\widehat{\textbf{z}}$,
with $\widehat{\bf z}$ given by
\begin{equation}
\widehat{\textbf{z}} = \arg \mathop{\min}_{\textbf{z} \in \mathbb{R}^d} \frac{1}{2}\| \textbf{y} - \textbf{AWz} \|_2^2  + \lambda \phi({\bf z}),
\label{eq:unconstrained_synthesis}
\end{equation}
where $\textbf{y}\in \mathbb{R}^n$ is a vector containing all the pixels (in lexicographic order)
of the observed image, ${\bf A} \in \mathbb{R}^{n\times n}$ is a matrix representation of the
(periodic) convolution, the columns of matrix ${\bf W}\in \mathbb{R}^{n\times d}$ ($d\geq n$)
are the elements of the adopted frame, $\phi$
is a regularizer encouraging $\widehat{\bf z}$ to be sparse (a typical choice, herein adopted,
is $\phi({\bf z}) = \|{\bf z}\|_1$), and $\lambda > 0$ is the
regularization parameter \cite{AfonsoFig_TIP10,EladMilanfar2007,Figueiredo03}.
The first term in \eqref{eq:unconstrained_synthesis} results from assuming the
usual observation model
\begin{equation}
{\bf y} = {\bf Ax} + {\bf n}\label{eq:observation_BC}
\end{equation}
where ${\bf x = Wz}$, and ${\bf n}$ is a sample of
white Gaussian noise with unit variance (any other value may be absorbed by $\lambda$).

Clearly, problem \eqref{eq:unconstrained_synthesis} has the form \eqref{eq:unconstrained2},
with $J=2$ and
\begin{align}
& g_1:\mathbb{R}^n \rightarrow \mathbb{R}, & & g_1({\bf v}) = \frac{1}{2}\|{\bf y}-{\bf v}\|_2^2\label{eq:g_1}\\
& g_2:\mathbb{R}^d \rightarrow \mathbb{R}, & & g_2({\bf z}) = \lambda \|{\bf z}\|_1\label{eq:g_2}\\
& {\bf H}^{(1)} \in \mathbb{R}^{n\times d}, & &   {\bf H}^{(1)} = {\bf AW}\label{eq:H_1_synth}\\
& {\bf H}^{(2)} \in \mathbb{R}^{d\times d}, &  &  {\bf H}^{(2)} = {\bf I}\label{eq:H_2_synth}.
\end{align}
The proximity operators of $g_1$ and $g_2$,
 key components of Algorithm \ref{alg:ADMM2} for solving \eqref{eq:unconstrained_synthesis},
have simple expressions,
\begin{align}
& \mbox{prox}_{g_1/\mu_1}({\bf v}) = \frac{{\bf y} + \mu_1 {\bf v}}{1+\mu_1}, \label{eq:prox_g_1}\\
& \mbox{prox}_{g_2/\mu_2}({\bf z}) = \mbox{soft}\bigl({\bf z},\lambda/\mu_2\bigr), \label{eq:prox_g_2}
\end{align}
where ``soft" denotes the soft-threshold function in \eqref{eq:soft}.
Line 4 of Algorithm \ref{alg:ADMM2} (the other key component) has
the form
\begin{equation}
{\bf z}_{k+1} \leftarrow
\Bigl( {\bf W}^* {\bf A}^* {\bf A}{\bf W} + \frac{\mu_2}{\mu_1} {\bf I}\Bigr)^{-1} \bigl(  {\bf W}^*{\bf A}^*
\bzeta^{(1)} + \frac{\mu_2}{\mu_1}\bzeta^{(2)}\bigr).\label{eq:inverse_synth2}
\end{equation}
As shown in \cite{AfonsoFig_TIP10}, if ${\bf W}$ corresponds to a Parseval
frame \cite{Mallat},
{\it i.e.}, if\footnote{In fact, the frame only needs to be tight  \cite{Mallat}, {\it i.e.}, ${\bf W}{\bf W}^* = \omega {\bf I}$;
without loss of generality, we assume $\omega  = 1$, which yields lighter notation.
Recently, a method for relaxing the tight frame condition was proposed in \cite{Pustelnik2012}.}
 ${\bf W}{\bf W}^* = {\bf I}$ (although possibly ${\bf W}^*{\bf W}\neq {\bf I}$),
the Sherman-Morrison-Woodbury  inversion formula can be used to write the matrix
inverse in \eqref{eq:inverse_synth2} as
\begin{equation}
\frac{\mu_1}{\mu_2} \Bigl({\bf I} -
{\bf W}^* \underbrace{ {\bf A}^* \bigl( {\bf A}{\bf A}^* + (\mu_2/\mu_1){\bf I}\bigr)^{-1} {\bf A}}_{{\bf F}} {\bf W} \Bigr).
\label{eq:inverse_synth}
\end{equation}
Assuming periodic BC, matrix ${\bf A}$ is circulant,
thus factors into
\begin{equation}
{\bf A} = {\bf U}^*\bLambda {\bf U},\label{eq:DFT_fact}
\end{equation}
where  ${\bf U}$ and ${\bf U}^*$ are the unitary matrices representing the DFT
and its inverse, and $\bLambda$ is the diagonal matrix of the DFT
coefficients of the convolution kernel.
Using \eqref{eq:DFT_fact}, matrix ${\bf F}$ defined in
\eqref{eq:inverse_synth} can be written as
\begin{equation}
{\bf F} = {\bf U}^* \bLambda^* \bigl( |\bLambda|^2 + (\mu_2/\mu_1){\bf I}\bigr)^{-1}\bLambda{\bf U},\label{eq:F_synth_FFT}
\end{equation}
where $|\bLambda|^2$ denotes the squared
absolute values of the entries of $\bLambda$. Since the product
$\bLambda^* \bigl( |\bLambda|^2 + (\mu_2/\mu_1){\bf I}\bigr)^{-1}\bLambda$
involves only diagonal matrices, it can be computed with only $O(n)$ cost, while the products by
${\bf U}$ and ${\bf U}^*$ (DFT and its inverse) can be computed with $O(n \log n)$ cost,
using the FFT; thus, computing ${\bf F}$ and multiplying by it has $O(n\log n)$ cost.

The leading cost of each application of \eqref{eq:inverse_synth2}
(line 4 of Algorithm \ref{alg:ADMM2}) is either the $O(n \log n)$
cost associated with ${\bf F}$ or the cost of the products by
${\bf W}^*$ and ${\bf W}$. For most tight frames used in image
restoration ({\it e.g.} undecimated wavelet transforms, curvelets,
complex wavelets), these products correspond to
direct and inverse transforms, for which fast $O(n\log n)$
algorithms exist \cite{Mallat}. In conclusion, under periodic BC
and for a large class of frames, each iteration of Algorithm~\ref{alg:ADMM2}
for solving \eqref{eq:unconstrained_synthesis} has $O(n \log n)$ cost.
Finally, convergence of this instance of Algorithm~\ref{alg:ADMM2} is established
in the following proposition.

\begin{prop} \label{th:prop2}The instance of Algorithm~\ref{alg:ADMM2}, with the
definitions in \eqref{eq:g_1}--\eqref{eq:H_2_synth}, with line 4 computed
as in \eqref{eq:inverse_synth2}, and the proximity operators in line
6 as given in \eqref{eq:prox_g_1} and \eqref{eq:prox_g_2}, converges
to a solution of \eqref{eq:unconstrained_synthesis}.
\end{prop}

\begin{proof}
Since $g_1$ and $g_2$ are proper, closed, convex functions,
so is the objective function in \eqref{eq:unconstrained_synthesis}.
Because $g_2$ is coercive, so is the objective function in \eqref{eq:unconstrained_synthesis},
thus its set of minimizers is not empty \cite{CombettesSIAM}.
Matrix ${\bf H}^{(2)} = {\bf I}$ obviously has full column rank, which implies
that ${\bf G}$ also has full column rank; that fact
combined with the fact that $g_1$ and $g_2$ are proper, closed, convex functions,
and that a minimizer exists, allows invoking Proposition~\ref{th:Eckstein2}
which concludes the proof.
\end{proof}

\subsection{Frame-Based Analysis Formulation}\label{sec:WAV_A_BC}
In the frame-based analysis approach, the problem is formulated
directly with respect to the unknown image (rather than its synthesis
coefficients),
\begin{equation}
\widehat{\bf x} = \arg \mathop{\min}_{{\bf x} \in \mathbb{R}^n} \frac{1}{2}\| \textbf{y} - \textbf{Ax} \|_2^2  + \lambda \|{\bf Px}\|_1,
\label{eq:unconstrained_analysis}
\end{equation}
where ${\bf A}$ is as in \eqref{eq:unconstrained_synthesis}
and ${\bf P}\in \mathbb{R}^{m\times n}$ ($m\geq n$)
is the analysis operator of a Parseval frame, {\it i.e.}, satisfying\footnote{As above,
the frame is only required to be tight, {\it i.e.}, ${\bf P}^*{\bf P} = \omega {\bf I}$;
since $\omega  = 1$ can be obtained simply by normalization, there is no loss of generality.}
 ${\bf P}^*{\bf P}={\bf I}$
(although possibly ${\bf P}{\bf P} \neq {\bf I}$, unless the frame is an orthonormal basis)
\cite{Mallat}. Problem \eqref{eq:unconstrained_analysis} has the form \eqref{eq:unconstrained2},
with $J=2$, $g_1$ and $g_2$ as given in \eqref{eq:g_1} and \eqref{eq:g_2}, respectively, and
\begin{align}
& {\bf H}^{(1)} \in \mathbb{R}^{n\times n}, & &   {\bf H}^{(1)} = {\bf A}\label{eq:H_1_analysis}\\
& {\bf H}^{(2)} \in \mathbb{R}^{m\times n}, &  &  {\bf H}^{(2)} = {\bf P}\label{eq:H_2_analysis}.
\end{align}

Since the proximity operators of $g_1$ and $g_2$ are the same as in the
synthesis case (see \eqref{eq:prox_g_1} and \eqref{eq:prox_g_2}), only
line 4 of Algorithm \ref{alg:ADMM2} has a different form:
\begin{equation}
{\bf z}_{k+1} \leftarrow
\Bigl( {\bf A}^* {\bf A} + \frac{\mu_2}{\mu_1}\, {\bf P}^*{\bf P}\Bigr)^{-1} \bigl(  {\bf A}^*
\bzeta^{(1)} + \frac{\mu_2}{\mu_1}\, {\bf P}^*\bzeta^{(2)}\bigr).\label{eq:invAnalysis}
\end{equation}
Since ${\bf P}^*{\bf P}={\bf I}$ and ${\bf A} = {\bf U}^*\bLambda {\bf U}$,
the matrix inverse is simply
\begin{equation}
\Bigl( {\bf A}^* {\bf A} + \frac{\mu_2}{\mu_1}\, {\bf I}\Bigr)^{-1} =
{\bf U}^*\bigl(|\bLambda|^2 + (\mu_2/\mu_1){\bf I}\bigr)^{-1} {\bf U},\label{eq:invAnalysisp1}
\end{equation}
which has $O(n\log n)$ cost, since $\bigl(|\bLambda|^2 + (\mu_2/\mu_1){\bf I}\bigr)$
is a diagonal matrix and the products by ${\bf U}$ and ${\bf U}^*$ (the DFT and its inverse)
can be computed using the FFT.

In conclusion, under periodic BC and for a large class of frames,
each iteration of Algorithm~\ref{alg:ADMM2} for solving
\eqref{eq:unconstrained_analysis} has $O(n \log n)$ cost.
Finally, convergence of this instance of Algorithm~\ref{alg:ADMM2} is claimed
in the following proposition.

\begin{prop}\label{th:prop_analysis}
 The instance of Algorithm~\ref{alg:ADMM2}, with the
definitions in \eqref{eq:g_1}, \eqref{eq:g_2}, \eqref{eq:H_1_analysis},
and \eqref{eq:H_2_analysis}, with line 4 computed
as in \eqref{eq:invAnalysis}--\eqref{eq:invAnalysisp1}, and the proximity operators in line
6 as given in \eqref{eq:prox_g_1} and \eqref{eq:prox_g_2}, converges
to a solution of \eqref{eq:unconstrained_analysis}.
\end{prop}

\begin{proof}
Since $g_1$ and $g_2$ are proper, closed, convex functions,
so is the objective function in \eqref{eq:unconstrained_synthesis}.
Because $g_2$ is coercive and ${\bf P}$ is the analysis operator of
a tight frame, $\ker({\bf P}) = \{\bzero\}$ (where $\ker({\bf P})$
is the null space of matrix ${\bf P}$), the objective
function in \eqref{eq:unconstrained_synthesis} is coercive,
thus its set of minimizers is not empty \cite{CombettesSIAM}.
Matrix ${\bf H}^{(2)} = {\bf P}$ has full column rank, thus the same
is true of ${\bf G}$; combining that with $g_1$ and $g_2$ being proper, closed, and convex,
and given the existence of a minimizer, allows invoking Proposition~\ref{th:Eckstein2}
to conclude the proof.
\end{proof}

\subsection{Total Variation}
\label{ssec:TV_Cyclic}
The classical formulation of TV-based image deconvolution consists in
a convex optimization problem
\begin{equation}
\widehat{\textbf{x}} = \arg \mathop{\min}_{\textbf{x} \in \mathbb{R}^n} \frac{1}{2}\| \textbf{y} - \textbf{Ax} \|_2^2  + \lambda \sum_{i=1}^n \|\textbf{D}_i \; \textbf{x}\|_2 ,
\label{eq:unconstrained_TV}
\end{equation}
where $\textbf{x}\in \mathbb{R}^n, \textbf{y} \in \mathbb{R}^n$, $\lambda >0 $, and ${\bf A}\in \mathbb{R}^{n\times n}$
have the same meanings as above, and each $\textbf{D}_i \in \mathbb{R}^{2\times n}$ is the matrix that
computes the horizontal and vertical differences at pixel $i$ (also with periodic BC)
\cite{Tao,WangZhang_ADMMTV_08}. The sum $\sum_{i=1}^n \| \textbf{D}_i \, \textbf{x}\|_2 = \mbox{TV}(\textbf{x})$ defines
the so-called {\it total-variation}\footnote{In fact, this is the so-called isotropic discrete approximation
of the continuous total-variation \cite{Chambolle2004,ROF}.} (TV) of image $\textbf{x}$.

Problem \eqref{eq:unconstrained_TV} has the form \eqref{eq:unconstrained2}, with $J=n+1$, and
\begin{align}
& g_1 : \mathbb{R}^n \rightarrow \mathbb{R}, & & g_1({\bf v}) = \frac{1}{2}\|{\bf y} - {\bf v}\|_2^2,\label{eq:corresp3}\\
& g_i : \mathbb{R}^2 \rightarrow \mathbb{R}, & &  g_i({\bf v}) = \lambda\, \|{\bf v}\|_2, \;\,\mbox{for}\;\, i=2,...,n+1,\label{eq:corresp3b}\\
& {\bf H}^{(1)}\in \mathbb{R}^{n\times n}, & & {\bf H}^{(1)} = {\bf A},\label{eq:corresp3c}\\
& {\bf H}^{(i)} \in \mathbb{R}^{2\times n}, & & {\bf H}^{(i)} = {\bf D}_{i-1} , \;\; \mbox{for}\;\; i=2,...,n+1,\label{eq:corresp3d}
\end{align}

For matrix $\bUpsilon$ (see \eqref{eq:diagonal_penalty}), we adopt $\mu_2 = \cdots = \mu_{n+1} > 0$.
The main steps of the resulting instance of Algorithm \ref{alg:ADMM2} for
solving \eqref{eq:unconstrained_TV} are as follows.

\begin{itemize}
  \item As shown in \cite{Tao,WangZhang_ADMMTV_08}, the matrix to be inverted in line 4 of
  Algorithm \ref{alg:ADMM2} can be written as
  \begin{multline}
  \mu_1 {\bf A}^* {\bf A} + \mu_2 \sum_{j=1}^n {\bf D}_j^* \; {\bf D}_j = \\
  \mu_1 {\bf A}^* {\bf A} + \mu_2 \bigl( ({\bf D}^h)^* {\bf D}^h + ({\bf D}^v)^* {\bf D}^v\bigr) ,\label{eq:matrix_periodic}
  \end{multline}
  where ${\bf D}^h,\, {\bf D}^v \in \mathbb{R}^{n\times n}$ are the matrices
  collecting the first and second rows, respectively, of each of the $n$
  matrices ${\bf D}_i$; that is, ${\bf D}^{h}$ computes all the horizontal
  differences and ${\bf D}^v$ all the vertical differences. With periodic BC, the
  matrices in \eqref{eq:matrix_periodic} can be factorized as
  \begin{equation}
   \textbf{A} = \textbf{U}^* \bLambda \textbf{U},\;\;\;\;  \textbf{D}^h   = \textbf{U}^* \bDelta^{\! h} \textbf{U}, \; \;\;\;
   \textbf{D}^v = \textbf{U}^* \bDelta^{\! v} \textbf{U},
   \end{equation}
   where $\bLambda$, $\bDelta^{\! h}$, and $\bDelta^{\! v}$ are diagonal matrices. Thus, the inverse of
   the matrix in \eqref{eq:matrix_periodic}  can be written as
  \begin{equation}
   \textbf{U}^* \bigl( \mu_1 |\bLambda|^2 + \mu_2 |\bDelta^{\! h}|^2 + \mu_2 |\bDelta^{\! v}|^2 \bigr)^{-1} {\bf U} \equiv {\bf K},\label{eq:fast_inversion}
  \end{equation}
  which involves a diagonal inversion, with $O(n)$ cost, and the products by $\textbf{U}$ and $\textbf{U}^*$
  (the direct and inverse DFT), which have $O(n\log n)$ cost (by FFT).
  \item The sum yielding the vector to which this inverse is applied  (line 4 of Algorithm~\ref{alg:ADMM2}) is
  \begin{eqnarray}
  \sum_{j=1}^{n+1} \mu_j (\textbf{H}^{(j)})^* \bzeta^{(j)}  & = & \mu_1 {\bf A}^* \bzeta^{(1)}
  + \mu_2 \sum_{j=1}^n {\bf D}_j^* \, \bzeta^{(j+1)} \nonumber \\
  & & \hspace {-2.5cm} = \mu_1 {\bf A}^* \bzeta^{(1)} + \mu_2 ({\bf D}^h)^* \bdelta^h + \mu_2 ({\bf D}^v)^* \bdelta^v,\nonumber
  \end{eqnarray}
  where $\bdelta^h,\, \bdelta^v \in \mathbb{R}^n$ contain the first and second component,
  respectively, of all the $\bzeta^{(j)}$ vectors, for $j=1,...,n$.
  As seen above, the product by ${\bf A}^*$ has $O(n\log n)$ cost via the FFT, while the
  products by $({\bf D}^h)^*$ and $({\bf D}^v)^*$ (which correspond to
  local differences) have $O(n)$ cost.
  \item The proximity operator of $g_{1}$ is as given in \eqref{eq:prox_g_1}.
  \item The proximity operator of $g_i ({\bf v}) = \lambda\, \|{\bf v}\|_2$, for $i=2,...,n+1$, is
  the so-called vector-soft function \cite{Tao,WangZhang_ADMMTV_08},
  \begin{eqnarray}
  \hspace{-0.2cm} \mbox{prox}_{\frac{g_i}{\mu_i}} ({\bf v}) = \mbox{v-soft}\Bigl( {\bf v} , \frac{\lambda}{\mu_i} \Bigr)
   = \frac{{\bf v}}{\|{\bf v}\|_2} \max\Bigl\{ \|{\bf v}\|_2 - \frac{\lambda}{\mu_i} , 0 \Bigr\}, \nonumber
  \end{eqnarray}
  with the convention that $\mbox{\boldmath $0$}/\|\mbox{\boldmath $0$}\|_2 = 0.$
\end{itemize}

Plugging all these equalities into Algorithm \ref{alg:ADMM2} yields Algorithm \ref{alg:ADMM3}, which
is similar to the one presented in \cite{Tao}.

\begin{algorithm}
\DontPrintSemicolon
\caption{ADMM-TV \label{alg:ADMM3}}
Set $k=0$, choose $\mu_1,\mu_2 > 0$, $\textbf{u}^{(j)}_0$, $\rule[-0.15cm]{0cm}{0.3cm}\textbf{d}^{(j)}_0$,  $j=1,...,n+1$\;
\Repeat{stopping criterion is satisfied}{
$\rule[-0.2cm]{0cm}{0.4cm} \bzeta \leftarrow \textbf{u}_k +\textbf{ d}_k$\\
$\bdelta^h = \bigl[ (\bzeta^{(2)})_1 , ..., (\bzeta^{(n+1)})_1 \bigr]^*\rule[-0.2cm]{0cm}{0.5cm}$\\
$\bdelta^v = \bigl[ (\bzeta^{(2)})_2 , ..., (\bzeta^{(n+1)})_2 \bigr]^*\rule[-0.3cm]{0cm}{0.5cm}$\\
$\rule[-0.3cm]{0cm}{0.6cm}\textbf{z}_{k+1} \leftarrow
{\bf K} \Bigl( \mu_1 {\bf A}^* \bzeta^{(1)} + \mu_2 ({\bf D}^h)^* \bdelta^h + \mu_2 ({\bf D}^v)^* \bdelta^v \Bigr)$\\							
$\rule[-0.3cm]{0cm}{0.6cm} \textbf{u}_{k+1}^{(1)} \leftarrow  \frac{1}{1 + \mu_1} \Bigl( {\bf y} + \mu_1 \bigl(
{\bf A} \; {\bf z}_{k+1} - {\bf d}_k^{(1)}\bigr)\Bigr) $\\
$\rule[-0.3cm]{0cm}{0.6cm} \textbf{d}_{k+1}^{(1)} \leftarrow \textbf{d}_k^{(1)} - ( \textbf{A} \; \textbf{z}_{k+1} - \textbf{u}_{k+1}^{(1)})$\\
\For{$j=2$ {\rm to} $n+1$}{$\rule[-0.3cm]{0cm}{0.6cm}
\textbf{u}_{k+1}^{(j)} \leftarrow \mbox{vect-soft}\Bigl( {\bf D}_{j-1} \; {\bf z}_{k+1} - {\bf d}_k^{(j)} , \frac{\lambda}{\mu_2} \Bigr) $ \\
$\rule[-0.3cm]{0cm}{0.6cm}\textbf{d}_{k+1}^{(j)} \leftarrow \textbf{d}_k^{(j)} - ( \textbf{D}_{j-1} \; \textbf{z}_{k+1} - \textbf{u}_{k+1}^{(j)}$ )}
$\rule[-0.2cm]{0cm}{0.4cm}k \leftarrow k+1$}
\end{algorithm}

Finally, convergence of Algorithm~\ref{alg:ADMM3} is addressed in the following
proposition.

\begin{prop} \label{th:propTV}
Consider problem \eqref{eq:unconstrained_TV} and assume that
$\bone \not\in \ker({\bf A})$. Then Algorithm~\ref{alg:ADMM3} converges
to a solution of \eqref{eq:unconstrained_TV}.
\end{prop}

\begin{proof}
Since $g_1, ...,g_{n+1}$ \eqref{eq:corresp3}--\eqref{eq:corresp3b}
are proper, closed, convex functions, so is the objective function in
\eqref{eq:unconstrained_TV}. That $\bone \not\in \ker({\bf A})$ implies
coercivity of the objective function in \eqref{eq:unconstrained_TV},
thus existence of minimizer(s), was shown in \cite{ChambolleLions1997}.
Matrix ${\bf G} = \bigl[ {\bf A}^*, {\bf D}_1^*,...,{\bf D}_n^*\bigr]^*$
can be written as a permutation of the rows of $\bigl[ {\bf A}^*, {\bf D}^*\bigr]^*$,
where ${\bf D} = \bigl[({\bf D}^h)^*,...,({\bf D}^v)^*\bigr]^*$;
since ${\bf Dx} = \bzero \Leftrightarrow {\bf x} = \alpha\bone$ ($\alpha\in\mathbb{R}$),
if $\bone\not\in\ker({\bf A})$, then ${\bf G}$ has full column rank.
Finally, since Algorithm~\ref{alg:ADMM3} is an instance of Algorithm~\ref{alg:ADMM2}
(with all updates computed exactly), Proposition \ref{th:Eckstein2}
guarantees its convergence to a solution of \eqref{eq:unconstrained_TV}.
\end{proof}

Notice that the condition $\bone \not\in \ker({\bf A})$ is very mild, and
clearly satisfied by the convolution with any low-pass-type filter, such as
those modeling uniform, motion, Gaussian, circular, and other types of
blurs. In fact, for these blurring filters, ${\bf A}\bone = \beta\bone$,
where $\beta$ is the DC gain (typically $\beta = 1$).

\section{Deconvolution with Unknown Boundaries}
\label{sec:proposed}
\subsection{The Observation Model}
\label{sec:proposed_model}
To handle images with unknown boundaries, we model
the boundary pixels as unobserved, which is achieved my modifying
\eqref{eq:observation_BC} into
\begin{equation}
{\bf y} = {\bf M\, A\, x} + {\bf n},\label{eq:model_unknown_BC}
\end{equation}
where ${\bf M} \in \{0,\, 1\}^{m \times n}$ (with $m < n$)  is a masking
matrix, {\it i.e.}, a matrix whose rows are a subset of the rows of an identity
matrix. The role of ${\bf M}$ is to observe only the subset of the
image domain in which the elements of ${\bf Ax}$ do not depend on the
boundary pixels; consequently, the BC assumed for
the convolution represented by ${\bf A}$ is irrelevant, and we may adopt
periodic BCs, for computational convenience.

Assuming that ${\bf A}$ models the convolution with a blurring filter with
a limited support of size $(1+2\, l) \times (1+2\, l)$, and that ${\bf x}$
and ${\bf Ax}$ represent square images of
dimensions $\sqrt{n} \times \sqrt{n}$, then  matrix ${\bf M} \in \mathbb{R}^{ m \times n}$,
with $m=(\sqrt{n} - 2\, l)^2$, represents the removal of a band of width $l$ of
the outermost pixels of the full convolved image ${\bf Ax}$.

Problem \eqref{eq:model_unknown_BC} can be seen as hybrid of deconvolution
and inpainting \cite{Chan2005}, where the missing pixels constitute the
unknown boundary. If  ${\bf M=I}$, model \eqref{eq:model_unknown_BC} reduces to a standard periodic deconvolution
problem. Conversely, if ${\bf A=I}$, \eqref{eq:model_unknown_BC} becomes a pure inpainting problem.
Moreover, the formulation \eqref{eq:model_unknown_BC}
can be used to model problems where not only the boundary, but also other
pixels, are missing, as in standard image inpainting.

The following subsections describe how   to handle observation
model \eqref{eq:model_unknown_BC}, in the context of the ADMM-based
deconvolution algorithms reviewed in the previous section.

\subsection{Frame-Based Synthesis Formulation}
\subsubsection{Mask Decoupling (MD)}
Under observation model \eqref{eq:model_unknown_BC}, the frame-based synthesis formulation
\eqref{eq:unconstrained_synthesis} changes to
\begin{equation}
\widehat{\textbf{z}} = \arg \mathop{\min}_{\textbf{z} \in \mathbb{R}^d} \frac{1}{2}\| \textbf{y} - \textbf{MAWz} \|_2^2  + \lambda \|{\bf z}\|_1.
\label{eq:unconstrained_synthesis_uBC}
\end{equation}
At this point, one could be tempted to map \eqref{eq:unconstrained_synthesis_uBC} into \eqref{eq:unconstrained2}
using the same correspondences as in \eqref{eq:g_1}, \eqref{eq:g_2},
and \eqref{eq:H_2_synth}, and simply changing  \eqref{eq:H_1_synth} into
\begin{align}
& {\bf H}^{(1)} \in \mathbb{R}^{m\times d}, & &   {\bf H}^{(1)} = {\bf MAW}.\label{eq:H_1_synth_uBC}
\end{align}
The problem with this choice is that the matrix to be inverted in line 4 of Algorithm~\ref{alg:ADMM2}
would become
\begin{equation}
\bigl( {\bf W}^* {\bf A}^*{\bf M}^*{\bf M} {\bf A}{\bf W} + (\mu_2/\mu_1)\, {\bf I}\bigr),\label{eq:inverse_synth_old}
\end{equation}
instead of the one in \eqref{eq:inverse_synth2}. The ``trick" used in Subsection~\ref{sec:synth_periodic}
to express this inversion in the DFT domain can no longer be used due to presence of ${\bf M}$,
invalidating the corresponding FFT-based implementation of line 4 of Algorithm~\ref{alg:ADMM2}.
It is clear that the source of the difficulty is the product ${\bf MA}$, which is the
composition of a mask (a spatial point-wise operation) with a circulant matrix (a point-wise
operation in the DFT domain); to  sidestep this difficulty, we need to decouple these two
operations, which is achieved by defining (instead of \eqref{eq:g_1})
\begin{align}
& g_1:\mathbb{R}^n \rightarrow \mathbb{R}, & & g_1({\bf v}) = \frac{1}{2}\|{\bf y}-{\bf M v}\|_2^2,\label{eq:g_1_uBC}
\end{align}
while keeping $g_2$, ${\bf H}^{(1)}$, and ${\bf H}^{(2)}$ as in \eqref{eq:g_2}--\eqref{eq:H_2_synth}.
With this choice, line 4 of Algorithm~\ref{alg:ADMM2} is still given by \eqref{eq:inverse_synth2} (with
its efficient FFT-based implementation); the only change is the proximity operator of the
new $g_1$,
\begin{align}
\mbox{prox}_{g_1/\mu_1}({\bf v}) & = \arg\min_{\bf u} \| {\bf Mu} - {\bf y} \|_2^2 + \mu_1 \| {\bf u} - {\bf v} \|_2^2\\
& = \bigl({\bf M}^*{\bf M} + \mu_1{\bf I}\bigr)^{-1}  \bigl({\bf M}^* {\bf y} + \mu_1 {\bf v} \bigr);
\label{eq:prox_g_1_uBC}
\end{align}
notice that, due to the special
structure of ${\bf M}$,  matrix ${\bf M}^*{\bf M}$ is diagonal, thus the inversion in
\eqref{eq:prox_g_1_uBC} has $O(n)$ cost, the same being true about the
product ${\bf M}^*{\bf y}$, which corresponds to extending the observed
image ${\bf y}$ to the size of ${\bf x}$, by creating a boundary of zeros
around it. Of course, both $\bigl({\bf M}^*{\bf M} + \mu_1{\bf I}\bigr)^{-1}$
and ${\bf M}^*{\bf y}$ can be pre-computed and then used throughout
the algorithm, as long as $\mu_1$ is kept constant. We refer to this approach
as {\it mask decoupling} (MD).

In conclusion, the proposed MD-based ADMM algorithm for image deconvolution with
unknown boundaries, under frame-based synthesis regularization, is Algorithm~\ref{alg:ADMM2}
with line 4 implemented  as in \eqref{eq:inverse_synth2} and the proximity operators
in line 6 given by  \eqref{eq:prox_g_1_uBC} and \eqref{eq:prox_g_2}. We refer to this
algorithm as FS-MD ({\it frame synthesis mask decoupling}).
As with the periodic BC, the leading cost is $O(n \log n)$ per iteration.
Finally, convergence of the FS-MD algorithm  is guaranteed by the following proposition
(the proof of which is similar to that of Proposition \ref{th:prop2}).

\begin{prop} \label{th:prop_synth_uBC}
Algorithm FS-MD ({\it i.e.}, the instance of Algorithm~\ref{alg:ADMM2}
with the definitions in \eqref{eq:g_1_uBC} and  \eqref{eq:g_2}--\eqref{eq:H_2_synth},
with line 4 computed as in \eqref{eq:inverse_synth2}, and the proximity operators in line
6 as given in \eqref{eq:prox_g_1_uBC} and \eqref{eq:prox_g_2}) converges
to a solution of \eqref{eq:unconstrained_synthesis_uBC}.
\end{prop}

\subsubsection{Using the Reeves-\v{S}orel Technique}
\label{sec:FS_CG}
An alternative to the approach just presented of decoupling
the convolution from the masking operator is to use the
method of \cite{Reeves05,Sorel12} to tackle the inversion \eqref{eq:inverse_synth_old}.
Following \cite{Reeves05}, notice that
\begin{equation}
{\bf AW} = {\bf S}\begin{bmatrix} {\bf MAW}\\ {\bf B}\end{bmatrix},\label{eq:Reeves1}
\end{equation}
where ${\bf B}$ contains the rows of ${\bf AW}$ that are missing from ${\bf MAW}$ (recall
that the rows of ${\bf M}$ are a subset of those of an identity matrix) and
${\bf S}$ is a permutation matrix that puts these missing rows in their original positions
in ${\bf AW}$. Noticing that
\begin{align}
{\bf W}^*{\bf A}^*{\bf AW} & = \bigl[ {\bf W}^*{\bf A}^*{\bf M}^* \; ,\; {\bf B}^*\bigr] {\bf S}^*{\bf S} \begin{bmatrix} {\bf MAW}\\ {\bf B}\end{bmatrix}\nonumber\\
& = {\bf W}^*{\bf A}^*{\bf M}^*{\bf MAW} + {\bf B}^*{\bf B}
\end{align}
(${\bf S}$ is a permutation matrix, thus ${\bf S}^*{\bf S}={\bf I}$),
the inverse of \eqref{eq:inverse_synth_old} can be written (with $\gamma = \mu_2/\mu_1$) as
\begin{align}
\left({\bf W}^*{\bf A}^*{\bf M}^*{\bf MAW} + \gamma {\bf I}\right)^{-1}  & =  \left( {\bf W}^*{\bf A}^*
{\bf AW} - {\bf B}^*{\bf B} + \gamma{\bf I}\right)^{-1} \nonumber\\
&\hspace{-1cm}  =  {\bf C} - {\bf C}{\bf B}^* \! \left( {\bf B} {\bf C} {\bf B}^*\!\! -
{\bf I} \right)^{-1} \! {\bf B}{\bf C} ,\label{eq:reeves}
\end{align}
where the second equality results from using the Sherman-Morrison-Woodbury matrix inversion identity,
after defining  ${\bf C} \equiv (\gamma {\bf I} + {\bf W}^*{\bf A}^* {\bf AW})^{-1}$. Since ${\bf A}$ is
circulant, ${\bf C}$ can be efficiently computed via FFT, as explained in \eqref{eq:inverse_synth2}--\eqref{eq:F_synth_FFT}.
The inversion $({\bf B} {\bf C} {\bf B}^*\! - {\bf I} )^{-1}$ in \eqref{eq:reeves} cannot be computed via FFT;
however, its dimension corresponds to the number of unknown boundary pixels (number of rows in
${\bf B}$), usually much smaller than image itself. As in \cite{Reeves05,Sorel12}, we use the
CG algorithm to solve the corresponding system; we confirmed experimentally that (as in
\cite{Sorel12}) taking only one CG iteration (initialized with the estimate of the previous
outer iteration) yields the fastest convergence, without degrading the final result.
Thus, in our experiments, we use a single CG iteration per outer iteration.

Approximately solving line 4 of Algorithm~\ref{alg:ADMM2} via one (or even
more) iterations of the CG algorithm, rather than an FFT-based exact solution,
makes  convergence more difficult to analyze, so we will not present a formal
proof. In a related problem \cite{FigBioucas_10}, it was shown experimentally
that if the iterative solvers used to implement the minimizations defining the
ADMM steps are warm-started ({\it i.e.}, initialized with the values from the
previous outer iteration), then the error sequences $\eta_k$ and $\rho_k$,
for $k=0,1,2,...$ are absolutely summable as required by Theorem~\ref{th:Eckstein}.
Finally, we refer to this algorithm as FS-CG ({\it frame synthesis conjugate gradient}).

\subsection{Frame-Based Analysis Formulation}
\subsubsection{Mask Decoupling (MD)}\label{sec:WAV_A_MD}
The frame-based analysis formulation corresponding to observation
model \eqref{eq:unconstrained_synthesis} is
\begin{equation}
\widehat{\bf x} = \arg \mathop{\min}_{{\bf x} \in \mathbb{R}^n} \frac{1}{2}\| {\bf y} - {\bf MAx} \|_2^2  + \lambda \|{\bf Px}\|_1.
\label{eq:unconstrained_analysis_uBC}
\end{equation}
Following the MD approach introduced for the synthesis formulation,
we map Problem \eqref{eq:unconstrained_analysis_uBC} into the form \eqref{eq:unconstrained2},
by using $g_1$ as defined in \eqref{eq:g_1_uBC}, and keeping ${\bf H}^{(1)}$, ${\bf H}^{(2)}$,
and $g_2$ as in the periodic BC case: \eqref{eq:H_1_analysis}, \eqref{eq:H_2_analysis}, and \eqref{eq:g_2},
respectively.

The only difference in the resulting instance of Algorithm~\ref{alg:ADMM2} is the use
of the proximity operator of the new $g_1$, as given in \eqref{eq:prox_g_1_uBC}.
In conclusion, the proposed ADMM-based algorithm for image deconvolution with
unknown boundaries, under frame-based analysis regularization, is simply
Algorithm~\ref{alg:ADMM2}, with line 4 implemented  as in \eqref{eq:invAnalysis}--\eqref{eq:invAnalysisp1},
and the proximity operators in line 6 given by \eqref{eq:prox_g_1_uBC} and
\eqref{eq:prox_g_2}. We refer to this algorithm as FA-MD ({\it frame analysis mask decoupling}).
The computational cost of the algorithm is $O(n \log n)$
per iteration, as in the periodic BC case. Convergence of FA-MD is
addressed by the following proposition (the proof of which is
similar to that of Proposition \ref{th:prop_analysis}).

\begin{prop} \label{th:prop_analysis_uBC}
Algorithm FA-MD ({\it i.e.}, Algorithm~\ref{alg:ADMM2} with the definitions in
 \eqref{eq:g_1_uBC}, \eqref{eq:H_1_analysis}, \eqref{eq:H_2_analysis}, and \eqref{eq:g_2},
with line 4 computed as in \eqref{eq:invAnalysis}--\eqref{eq:invAnalysisp1},
and the proximity operators in line 6 as given in \eqref{eq:prox_g_1_uBC} and \eqref{eq:prox_g_2})
converges to a solution of \eqref{eq:unconstrained_analysis_uBC}.
\end{prop}

\subsubsection{Using the Reeves-\v{S}orel Technique}
\label{sec:WAV_A_CG}
Consider Problem \eqref{eq:unconstrained_analysis_uBC} and map into the
form \eqref{eq:unconstrained2} using \eqref{eq:g_1}--\eqref{eq:g_2}, \eqref{eq:H_2_analysis},
and
\begin{equation}
{\bf H}^{(1)} \in \mathbb{R}^{n\times n},\; \;\; \; \;\;   {\bf H}^{(1)} = {\bf MA}.\label{eq:H_1_analysis_sorel}
\end{equation}
The matrix inverse  computed in line 4 of Algorithm~\ref{alg:ADMM2} is now (with $\gamma = \mu_2/\mu_1$)
\begin{equation}
\Bigl( {\bf A}^*{\bf M}^* {\bf M A} + \gamma\, {\bf P}^*{\bf P}\Bigr)^{-1} =
\Bigl( {\bf A}^*{\bf M}^* {\bf M A} + \gamma\, {\bf I}\Bigr)^{-1},\label{eq:invAnalysis_sorel}
\end{equation}
which can no longer be computed as in \eqref{eq:invAnalysisp1}, since matrix
${\bf M A}$ is not circulant. Using the same steps as
in \eqref{eq:Reeves1}--\eqref{eq:reeves}, with ${\bf A}$ replacing ${\bf AW}$
and ${\bf C} \equiv (\gamma {\bf I} + {\bf A}^* {\bf A})^{-1}$, leads to
\begin{equation}
\left({\bf A}^*{\bf M}^*{\bf MA} + \gamma {\bf I}\right)^{-1}  =
{\bf C} - {\bf C}{\bf B}^* \! \left( {\bf B} {\bf C} {\bf B}^*\!\! - {\bf I} \right)^{-1} \! {\bf B}{\bf C} ,\label{eq:reeves3}
\end{equation}
Since ${\bf A}$ is circulant, ${\bf C} = (\gamma {\bf I} + {\bf A}^* {\bf A})^{-1}$
can be efficiently computed via FFT, as in \eqref{eq:invAnalysisp1}. As in the
synthesis case, the inverse $({\bf B} {\bf C} {\bf B}^*\! - {\bf I} )^{-1}$ in \eqref{eq:reeves3}
is computed approximately by taking one (warm-started) CG iteration (for the reason
explained in Subsection~\ref{sec:FS_CG}).
We refer to the resulting algorithm as FA-CG ({\it frame analysis conjugate gradient}).

\subsection{TV-Based Deconvolution with Unknown Boundaries}
\subsubsection{Mask Decoupling (MD)}
\label{sec:TV_MD}
Given the observation model in \eqref{eq:model_unknown_BC}, TV-based
deconvolution with unknown boundaries is formulated as
\begin{equation}
\widehat{\textbf{x}} = \arg \min_{\textbf{x} \in \mathbb{R}^n} \frac{1}{2}\| \textbf{y} - {\bf M\, A\, x} \|_2^2  + \lambda \sum_{i=1}^n \|\textbf{D}_i \; \textbf{x}\|_2 ,
\label{eq:unconstrained_TV_uBC}
\end{equation}
with every element having the exact same meanings as above.

Following the MD approach, we map \eqref{eq:unconstrained_TV_uBC} into the form \eqref{eq:unconstrained2}
using $g_1$ as defined in \eqref{eq:g_1_uBC} and keeping all the other correspondences
\eqref{eq:corresp3b}--\eqref{eq:corresp3d} unchanged.
The only resulting change to Algorithm \ref{alg:ADMM3} is in line 7, which becomes
\begin{equation}
\textbf{u}_{k+1}^{(1)} \leftarrow   \bigl({\bf M}^*{\bf M} + \mu_1{\bf I}\bigr)^{-1}
\Bigl( {\bf M}^* {\bf y} + \mu_1 \bigl(
{\bf A} \; {\bf z}_{k+1} - {\bf d}_k^{(1)}\bigr)\Bigr), \label{eq:new_line}
\end{equation}
as a consequence of the application of the proximity operator of $g_1$, as given in
\eqref{eq:prox_g_1_uBC}. In summary, the ADMM-based algorithm for TV-based deconvolution
with unknown boundaries has the exact same structure as Algorithm~\ref{alg:ADMM3},
with line 7 replaced by \ref{eq:new_line}. We refer to this algorithm as TV-MD
({\it TV mask decoupling}). Finally, convergence of TV-MD is addressed in
the following proposition (the proof of which is similar to that of
Proposition~\ref{th:propTV}).

\begin{prop}
If $\bone\not\in\ker({\bf MA})$, then TV-MD  ({\it i.e.}, the version of
Algorithm~\ref{alg:ADMM3} with line 7 replaced by \eqref{eq:new_line})
converges to a solution of \eqref{eq:unconstrained_TV_uBC}.
\end{prop}

As for periodic BCs, the condition $\bone\not\in\ker({\bf MA})$
is very mild, and is satisfied by the convolution with any low-pass filter
(for which ${\bf A}\bone = \beta\bone$, where $\beta$ is the DC gain, typically $\beta\simeq 1$) and
${\bf M}\bone \neq 0$, if $m>1$ (at least one pixel is observed).

\vspace{0.4cm}
\subsubsection{Using the Reeves-\v{S}orel Technique}
\label{sec:TV_CG}
To use the Reeves-\v{S}orel technique to handle \eqref{eq:unconstrained_TV_uBC},
the mapping to the form \eqref{eq:unconstrained2} is done as in \eqref{eq:corresp3},
\eqref{eq:corresp3b}, and \eqref{eq:corresp3d}, with \eqref{eq:corresp3c} replaced
by \eqref{eq:H_1_analysis_sorel}. The consequence is a simple
modification of Algorithm~\ref{alg:ADMM3}, where ${\bf A}$ is replaced by ${\bf MA}$,
in line 8, while in line 6,  ${\bf A}^*$ is replaced by ${\bf A}^*{\bf M}^*$ and
${\bf K}$ is redefined (instead of \eqref{eq:fast_inversion}) as the inverse of the
matrix in \eqref{eq:matrix_periodic}, {\it i.e.} (with $\gamma = \mu_2/\mu_1$),
\begin{equation}
{\bf K} = \frac{1}{\mu_1}\Bigl( {\bf A}^*{\bf M}^*{\bf MA} + \gamma
 ({\bf D}^h)^*{\bf D}^h + \gamma ({\bf D}^v)^*{\bf D}^v \Bigr)^{-1} .
\end{equation}
A derivation similar to \eqref{eq:Reeves1}--\eqref{eq:reeves} allows writing
\begin{equation}
{\bf K}  = \frac{1}{\mu_1} \Bigl(  {\bf C} - {\bf C}{\bf B}^* \! \left( {\bf B} {\bf C} {\bf B}^*\!\! - {\bf I} \right)^{-1} \! {\bf B}{\bf C}\Bigr),
\label{eq:ReevesTV}
\end{equation}
where
\begin{equation}
{\bf C} = \bigl( {\bf A}^*{\bf A} + \gamma ({\bf D}^h)^*{\bf D}^h + \gamma ({\bf D}^v)^*{\bf D}^v\bigr)^{-1}.
\end{equation}
The circulant nature of ${\bf A}$, ${\bf D}^h$, and ${\bf D}^v$ allows computing ${\bf C}$
efficiently in the DFT domain (similarly to \eqref{eq:fast_inversion}). Finally, as above,
the inverse  $({\bf B} {\bf C} {\bf B}^*\! - {\bf I} )^{-1}$ in \eqref{eq:ReevesTV}
is approximated by taking, as above, one CG iterations; the resulting algorithm
is referred to as TV-CG ({\it TV conjugate gradient}).

\section{Experiments}
\label{sec:Exp}

In the experiments herein reported, we use the benchmark images {\it Lena} and
{\it cameraman} (of size $256\times 256$), with 4 different blurs (out-of-focus, linear motion,
uniform, and Gaussian), all of size $19\times19$ ({\it i.e.}, $(2\, l + 1)\times(2\, l + 1)$,
with $l=9$, as explained in  Section~\ref{sec:proposed_model}), at four different
BSNRs (\textit{blurred signal to noise ratio}): 30dB, 40dB, 50dB, and 60dB.
The reason why we concentrate on large blurs is that the effect of the boundary
conditions is very evident in this case.

To set up a scenario of unknown boundaries, the observed images (of size
$238\times 238$, since  $238 = 256 - 18$) are obtained by applying the
filters to the  $238\times 238$ central region of the original
images, using the band of pixels (of width 9) around this region as the boundary pixels.
In the proposed formulation, these boundary pixels are modeled as unknown. Notice that this procedure
is equivalent to convolving the full ($256\times 256$) images using an arbitrary
BC (periodic, for computational convenience) and then keeping only the valid region
({\it i.e.}, the one that does not depend on the BC). Since  $n = 256^2 = 65536$
and $m=238^2=56644$, the number of unknown boundary pixels is 8892.

Preliminary tests showed that both TV and frame-based analysis regularization
slightly (but consistently) outperform the frame-based synthesis formulation,
in terms of ISNR (\textit{improvement in signal to noise ratio}),
thus we will not report experiments with the latter. Concerning TV regularization,
 we consider four algorithms: (i) TV-MD (proposed in Subsection~\ref{sec:TV_MD}); (ii)
TV-CG (proposed in Subsection~\ref{sec:TV_CG}); (iii)
Algorithm \ref{alg:ADMM3} with the (wrong) assumption of periodic BC (referred
to as TV-BC); and (iv) Algorithm \ref{alg:ADMM3} also with the (wrong) assumption
of periodic BC, after pre-processing the observed image with the ``edgetaper" (ET) function
(referred to as TV-ET). Similarly, in  the frame-based\footnote{In the frame-based approach,
we use a  redundant Haar frame, with four levels, although we could use any other
frame with fast transforms.} analysis formulation, we  consider four algorithms:
(i) FA-MD (proposed in Subsection~\ref{sec:WAV_A_MD}); (ii) FA-CG
(proposed in Subsection~\ref{sec:WAV_A_CG}); (iii) the algorithm defined in
Subsection~\ref{sec:WAV_A_BC}, which wrongly assumes periodic BC (referred to
as FA-BC); and (iv) FA-BC, after pre-processing the observed image with the
``edgetaper" (ET) function (referred to as FA-ET).
Notice that the methods that assume periodic BC recover images of the same size
as the observed image ($\sqrt{m}\times\sqrt{m}$), while those based on
model \eqref{eq:model_unknown_BC} recover full images ($\sqrt{n}\times\sqrt{n}$),
extended with the estimated boundary pixels.

The regularization parameter $\lambda$ was manually adjusted to yield the best
ISNR for each experimental condition. For the ADMM parameters ($\mu_1$ and $\mu_2$),
which affect the convergence speed, we used heuristic rules, which lead to a good
performance of the algorithm: $\mu_2=10\lambda$ and $\mu_1 =\min\{1,5000\mu_1\}$.

%\begin{figure}[h]%[tb]
%	\centering
%	 \begin{tabular}{c@{  }c@{  }c@{  }c@{  }c@{  }c}
%   \includegraphics[width=0.055\textheight]{Blur1.eps} & \includegraphics[width=0.055\textheight]{Blur2.eps} &
%   \includegraphics[width=0.055\textheight]{Blur3.eps} & %\includegraphics[width=0.055\textheight]{Blur4.eps} &
%   %\includegraphics[width=0.055\textheight]{Blur5.eps} &
%    \includegraphics[width=0.055\textheight]{Blur7.eps}
%    	 \end{tabular}
%\caption{$9\times9$ blurring kernels used in the synthetic experiments (from left to right): out-of-focus,
%linear motion blur, square blur and Gaussian blur.}
%\label{fig:filtros}
%\end{figure}

Fig.~\ref{fig:Lena_uniform} shows the results obtained with the Lena image,
with uniform blur at 60dB BSNR, using frame-based analysis regularization.
Notice how the wrong assumption of periodic BC
causes a complete failure of the deblurring algorithm, while treating the
boundary as unknown allows the algorithm to obtain a remarkably good estimate
of the unobserved boundary pixels, as well as a deconvolved image without any
boundary artifacts. Fig.~\ref{fig:Cameraman_motion} illustrates
a similar experiment, now with the {\it cameraman} image and linear motion
blur at 40dB BSNR; the conclusions are similar to those drawn from
Fig.~\ref{fig:Lena_uniform}, with exception that, in this case, the
edgetaper-based approach is able to yield a reasonable image estimate,
although far from those produced by TV-CG and TV-MD.

\begin{figure}[th]%[tb]
	\centering
	 \begin{tabular}{cc}
   \includegraphics[width=0.16\textheight]{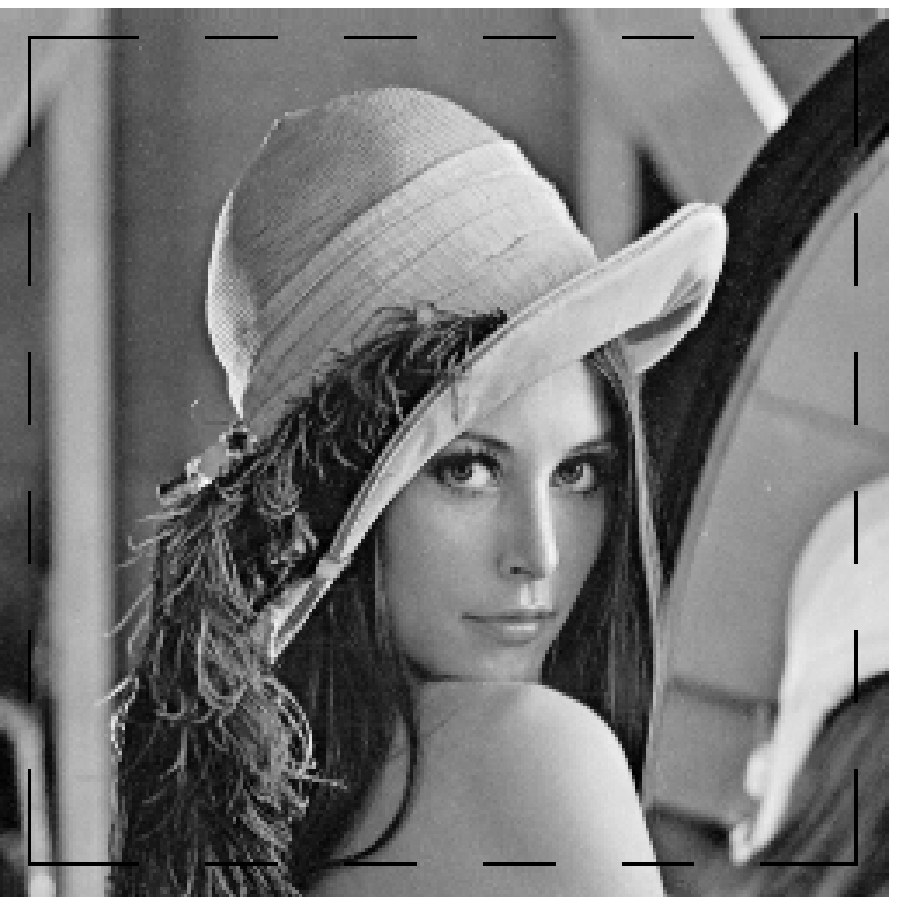} &
   \includegraphics[width=0.16\textheight]{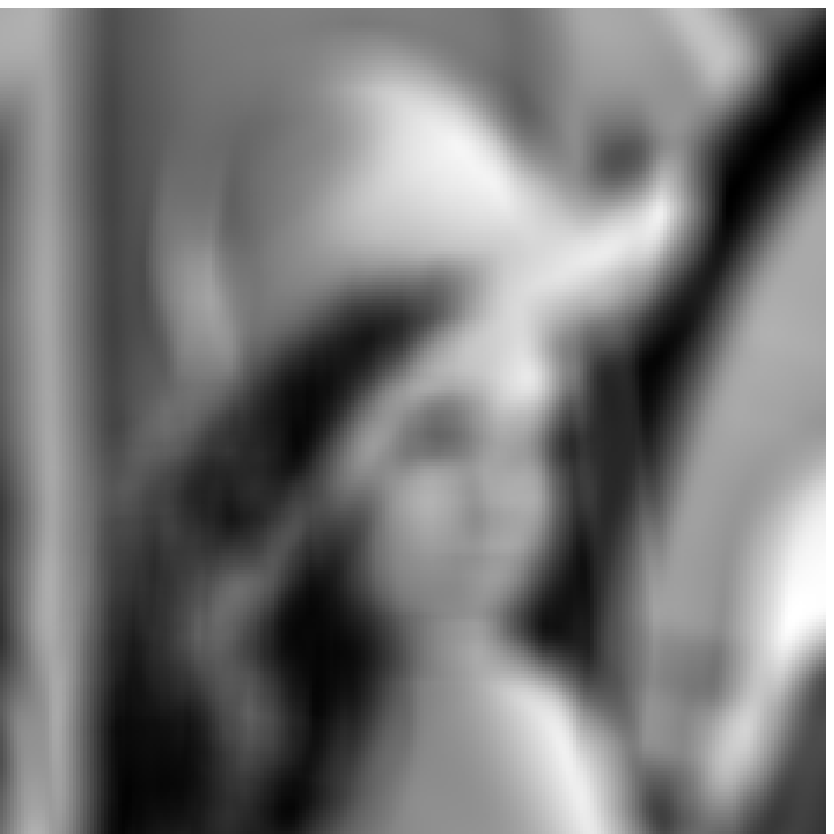} \\
   \vspace{0.15cm} {\small original  ($256\times 256$)}&  {\small observed ($238\times 238$)}\\
   \includegraphics[width=0.16\textheight]{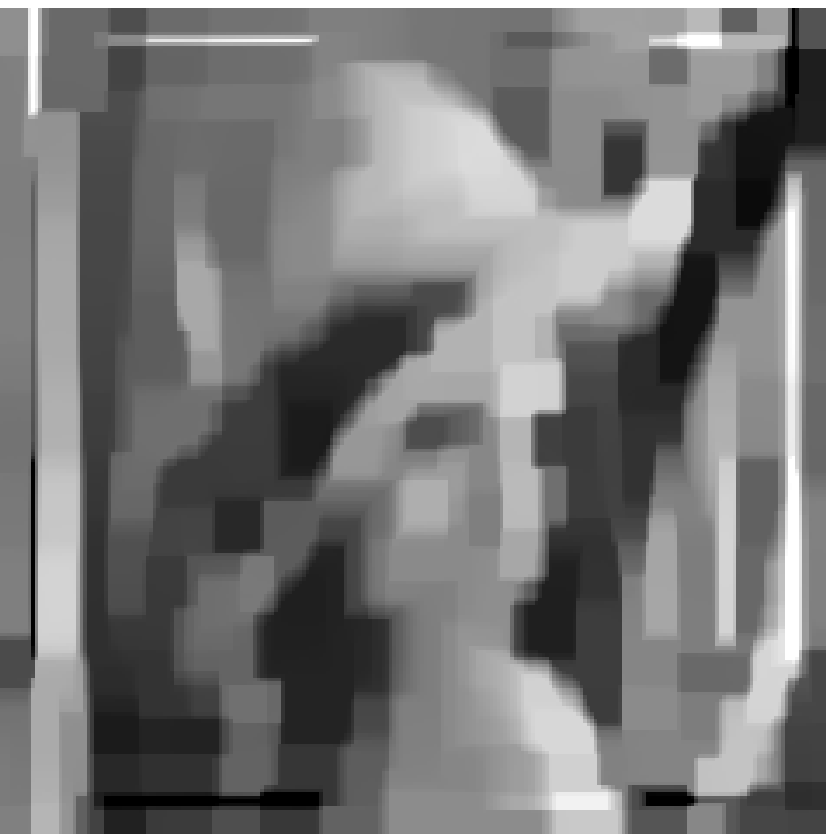} &
   \includegraphics[width=0.16\textheight]{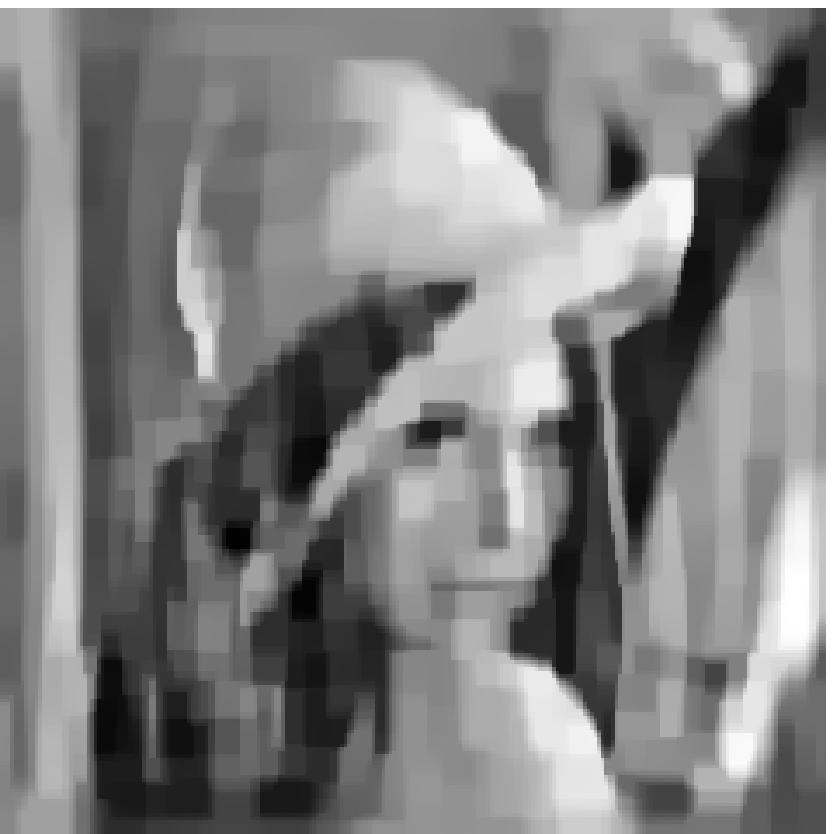} \\
   \vspace{0.25cm} {\small FA-BC (ISNR = -2.52dB)} &  {\small FA-ET (ISNR = 3.06dB)}  \\
  \includegraphics[width=0.16\textheight]{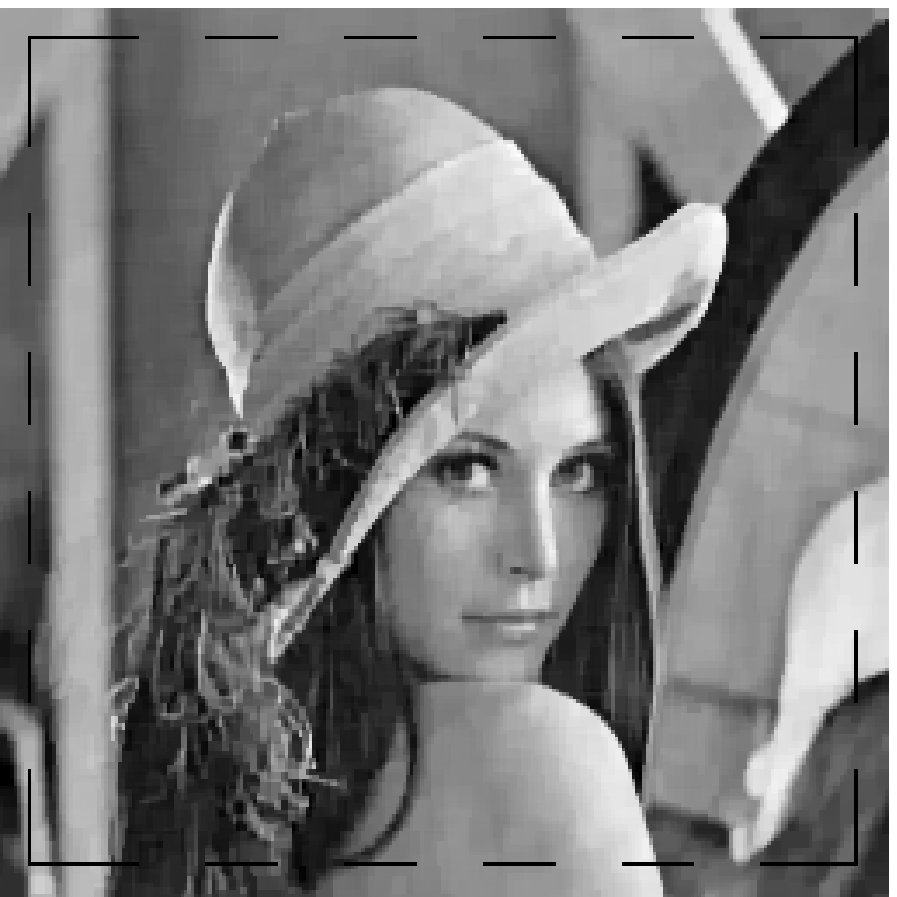}&
    \includegraphics[width=0.16\textheight]{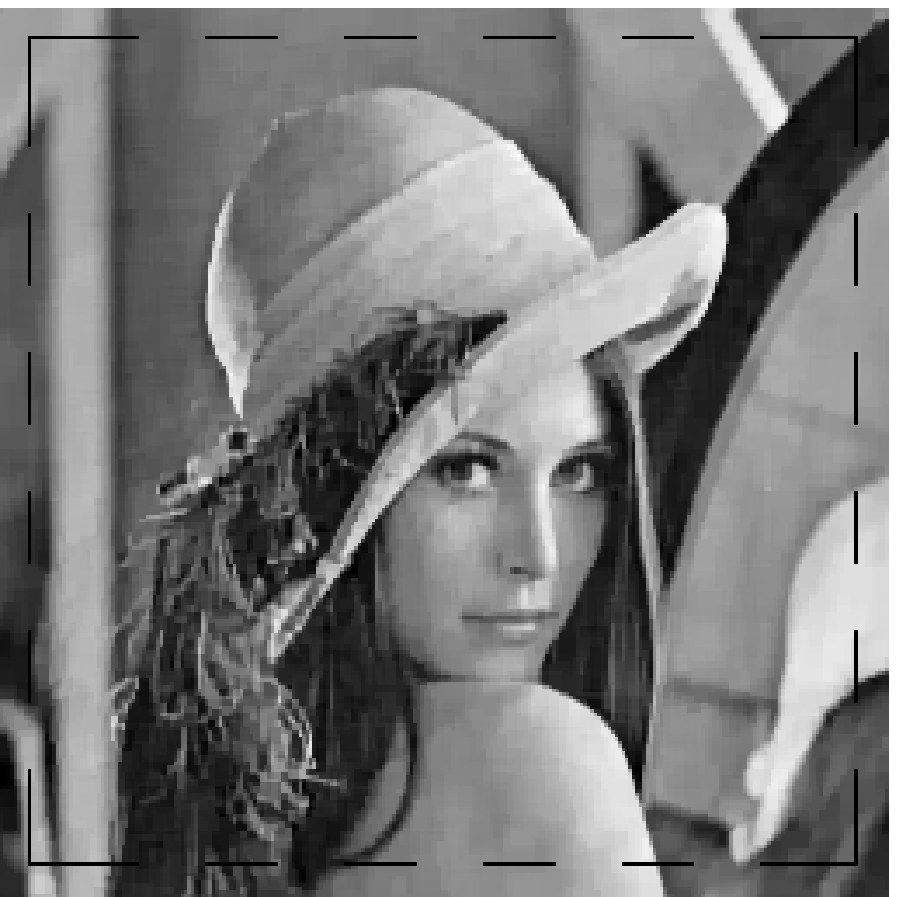} \\
   \vspace{0.15cm} {\small FA-CG (ISNR = 10.21dB)} &  {\small FA-MD (ISRN = 10.63dB)}
   	 \end{tabular}
\caption{Results obtained on the {\it Lena} image, degraded by a uniform $19 \times 19$ blur, at 60dB BSNR,
by the four algorithms considered (see text for the acronyms). Notice that FA-BC and FA-ET produce
$238\times 238$ images, while FA-CG and FA-MD yield $256\times 256$ images; the dashed square
shows the limit of the boundary region.}
\label{fig:Lena_uniform}
\end{figure}

\begin{figure}[h]%[tb]
	\centering
	 \begin{tabular}{cc}
   \includegraphics[width=0.16\textheight]{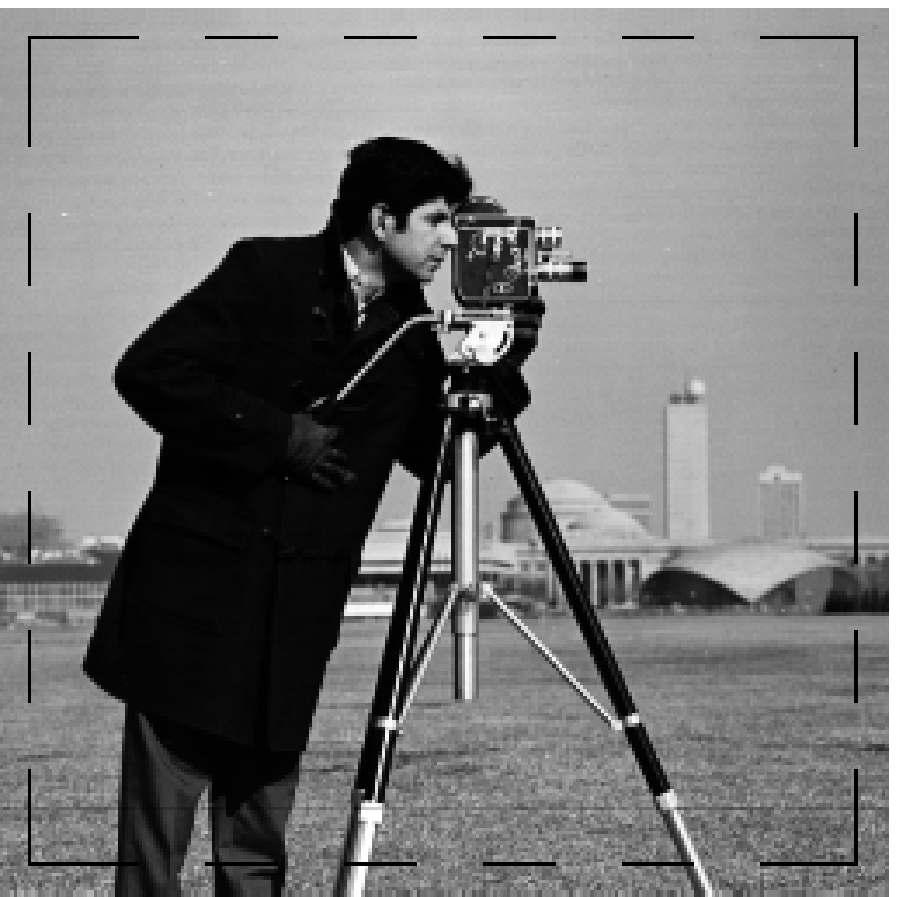} &
   \includegraphics[width=0.16\textheight]{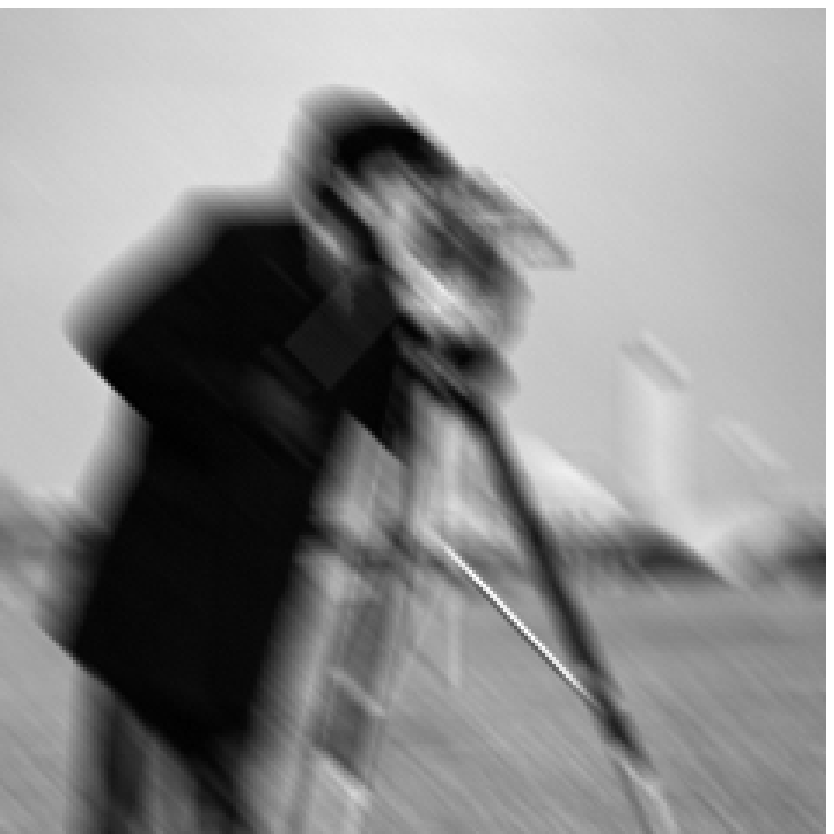} \\
   \vspace{0.15cm} {\small original  ($256\times 256$)}&  {\small observed ($238\times 238$)}\\
   \includegraphics[width=0.16\textheight]{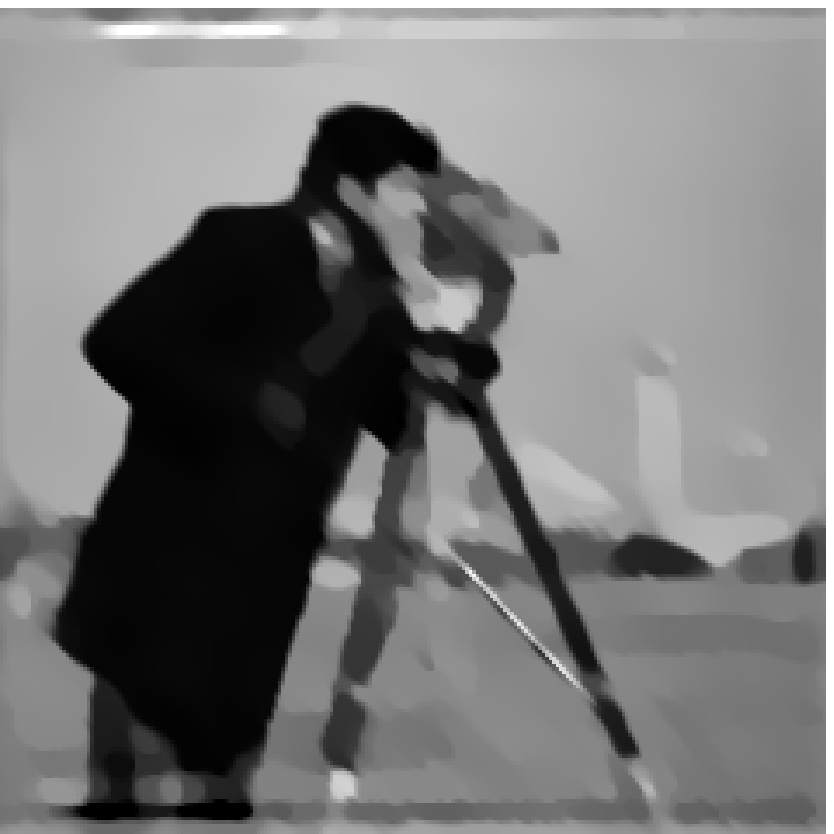} &
   \includegraphics[width=0.16\textheight]{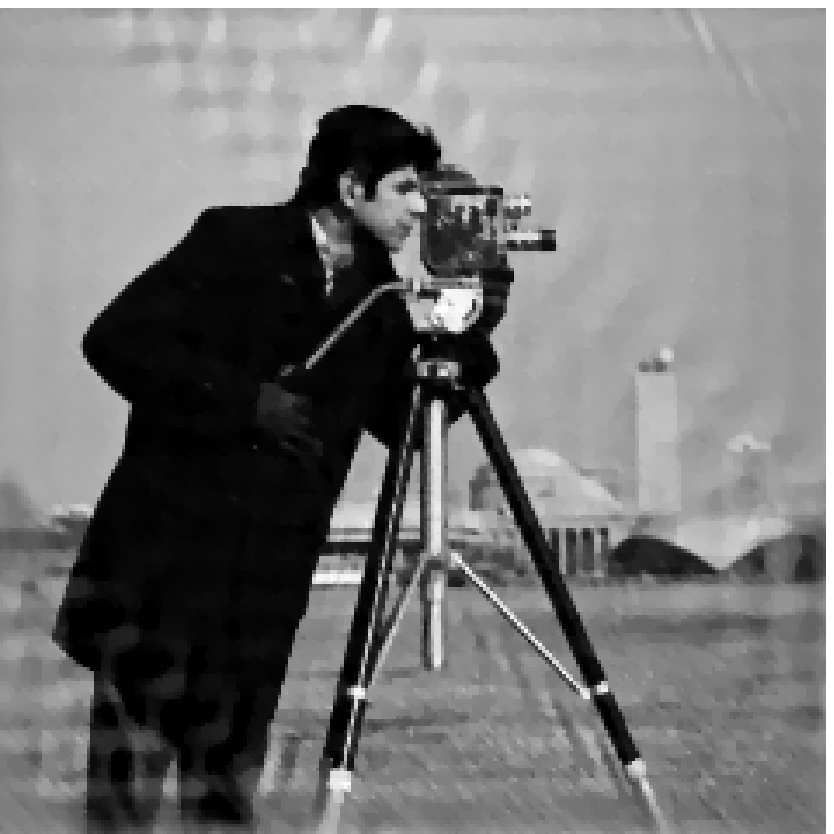} \\
   \vspace{0.25cm}  {\small TV-BC (ISNR = 0.91dB)} &  {\small TV-ET  (ISNR = 9.38dB) } \\
  \includegraphics[width=0.16\textheight]{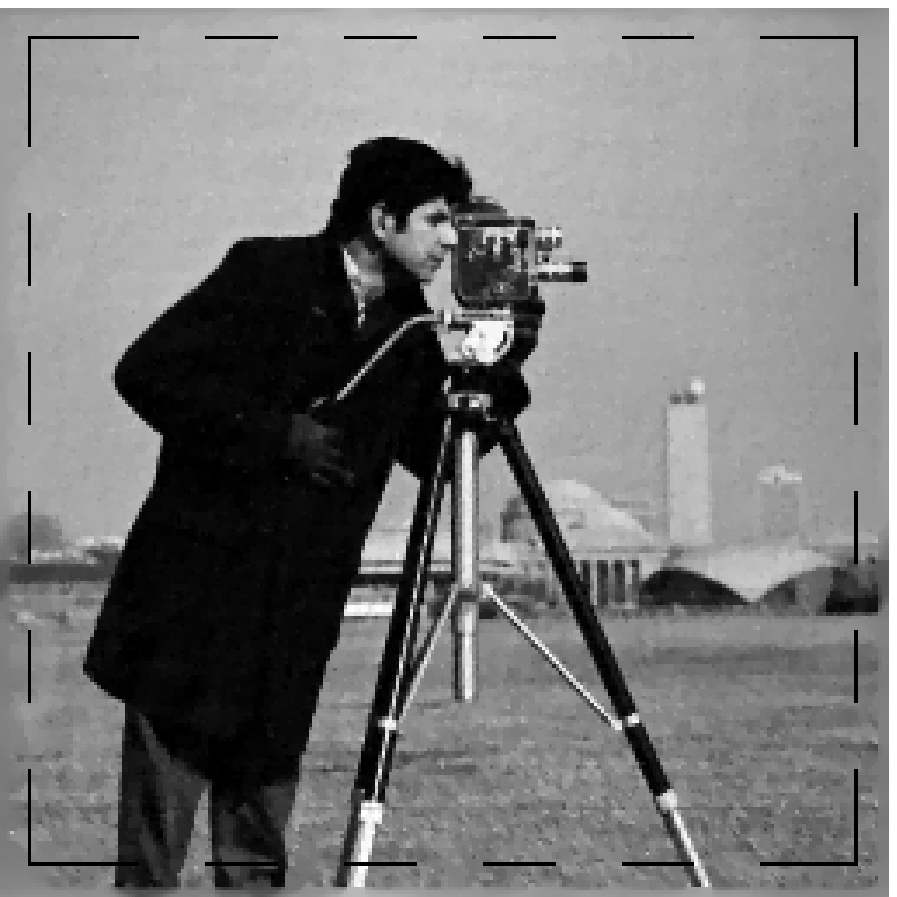}&
    \includegraphics[width=0.16\textheight]{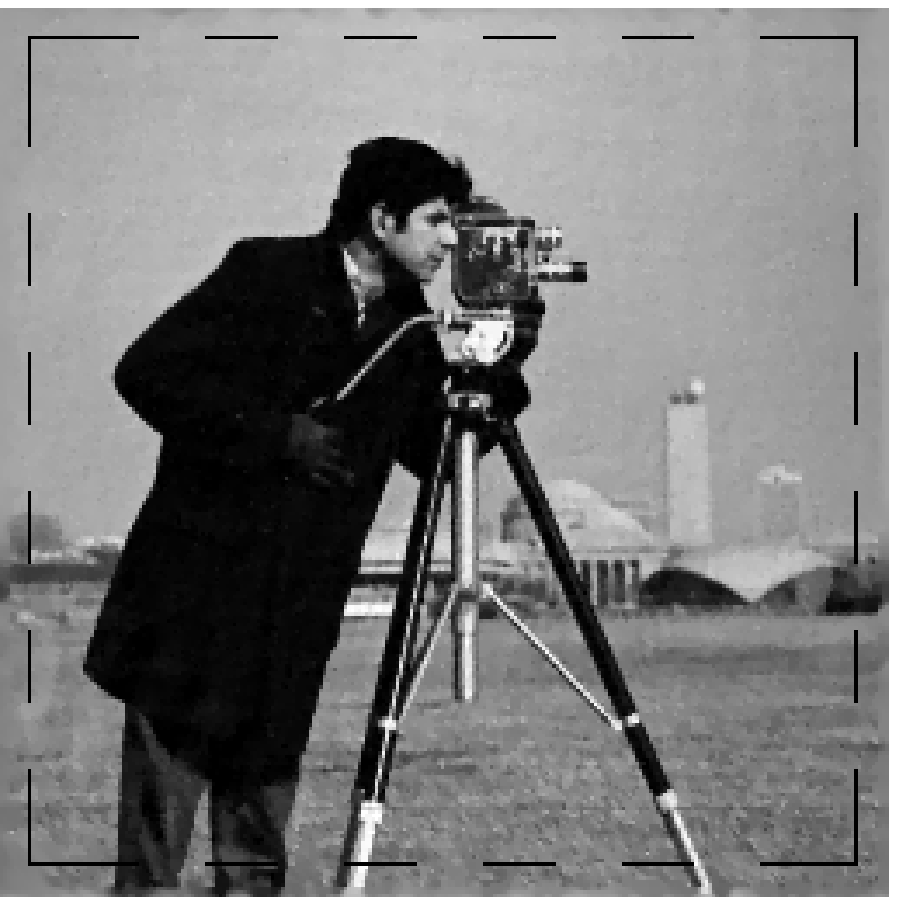} \\
   \vspace{0.15cm}  {\small TV-CG  (ISNR = 12.58dB)} &  {\small TV-MD  (ISNR = 12.59dB)}
   	 \end{tabular}
\caption{Results obtained on the {\it cameraman} image, degraded by a $19 \times 19$ linear motion
blur, at 40dB BSNR, by the four algorithms considered (see text for the acronyms).
Notice that FA-BC and FA-ET produce $238\times 238$ images, while FA-CG and FA-MD yield $256\times 256$
images; the dashed square shows the limit of the boundary region.}
\label{fig:Cameraman_motion}
\end{figure}

Tables~\ref{tab:TV_ISNR} and \ref{tab:WAV_ISNR} show the ISNR values achieved
by the four algorithms mentioned above, for TV and frame-based analysis
regularization, respectively. Naturally, the methods that treat the boundary pixels
as unknown (the CG and MD versions) yield very similar ISRN values,  since they simply
corresponds to two variants of an algorithm minimizing the same objective function.
It is also clear that by wrongly assuming a periodic BC (algorithms TV-BC
and FA-BC) leads to very poor results, as already illustrated in Figs.\ref{fig:Lena_uniform}
and \ref{fig:Cameraman_motion}, and that the edgetaper method is able to alleviate
this problem somewhat. Finally, we observe that frame-based analysis regularization
yields marginally better results than TV regularization.

\begin{table*}[htbp]
  \centering
  \caption{ISNR values (in {\rm dB}) obtained by the four algorithms tested (see the acronyms in the text)
  for TV-based image deblurring.}
    \begin{tabular}{lr|cccc|cccc|}
    %\toprule
         				&		 & \multicolumn{4}{c|}{\bf Cameraman}  & \multicolumn{4}{c|}{\bf Lena} \\
    %\midrule
   \multicolumn{1}{c}{\bf Blur} & {\bf BSNR}     & {TV-BC} &  {TV-ET}   & {TV-CG}    & {TV-MD}
    & {TV-BC} &  {TV-ET}   & {TV-CG}    & {TV-MD}  \\
   \hline
    Uniform & 30dB   & 1.05  & 4.42  & 5.47  & 5.44  & 0.44  & 3.07  & 5.52  & 5.26 \\
    Out-of-focus & 30dB   & 1.41  & 5.18  & 5.68  & 5.66  & 0.70  & 4.32  & 6.02  & 5.83 \\
    Linear motion & 30dB   & 2.09  & 6.83  & 8.26  & 8.24  & 1.17  & 5.88  & 7.85  & 7.54 \\
    Gaussian & 30dB   & 1.63  & 3.25  & 3.21  & 3.21  & 0.82  & 3.31  & 3.08  & 3.03 \\
    \hline
    \multicolumn{1}{r}{Average for} & 30dB  & 1.54 & 4.92 & 5.65 & 5.64 & 0.78 & 4.14 & 5.61 & 5.42\\
   \hline\hline
   Uniform & 40dB   & 1.04  & 4.89  & 7.07  & 7.02  & 0.44  & 3.13  & 7.18  & 7.05 \\
    Out-of-focus & 40dB   & 1.40  & 6.60  & 8.34  & 8.34  & 0.70  & 4.62  & 8.25  & 7.94 \\
     Linear motion & 40dB    & 2.07  & 8.30  & 12.57 & 12.41 & 1.16  & 6.31  & 11.24 & 11.08 \\
    Gaussian & 40dB   & 1.61  & 4.01  & 4.03  & 4.03  & 0.80  & 3.84  & 3.85  & 3.82 \\
 \hline
   \multicolumn{1}{r}{Average for} & 40dB  & 1.53 & 5.95 & 8.02 & 7.95 & 0.77 & 4.47 & 7.63 & 7.47\\
   \hline\hline
    Uniform & 50dB   & 1.04  & 5.82  & 9.77  & 9.75  & 0.44  & 3.14  & 9.09  & 9.06 \\
    Out-of-focus & 50dB  & 1.40  & 7.35  & 11.85 & 11.78 & 0.70  & 4.67  & 10.92 & 10.80 \\
    Linear motion & 50dB   & 2.06  & 8.64  & 16.87 & 16.67 & 1.17  & 6.37  & 12.33 & 12.59 \\
    Gaussian & 50dB   & 1.61  & 4.44  & 4.77  & 4.78  & 0.79  & 4.00  & 4.48  & 4.47 \\
   \hline
    \multicolumn{1}{r}{Average for} & 50dB  & 1.53 & 6.56 & 10.81 & 10.77 & 0.77 & 4.54 & 9.20 & 9.23\\
   \hline\hline
    Uniform & 60dB    & 1.04  & 6.24  & 11.59 & 11.95 & 0.44  & 3.14  & 9.84  & 10.41 \\
    Out-of-focus & 60dB   & 1.40  & 7.51  & 14.68 & 14.89 & 0.70  & 4.67  & 12.82 & 13.27 \\
    Linear motion & 60dB   & 2.07  & 8.69  & 20.03 & 19.88 & 1.16  & 6.38  & 13.71 & 13.56 \\
    Gaussian & 60dB   & 1.53  & 4.67  & 4.91  & 4.97  & 0.80  & 4.02  & 4.60  & 4.70 \\
    \hline
    \multicolumn{1}{r}{Average for} & 60dB  & 1.53 & 6.78 & 12.80 & 12.92 & 0.77 & 4.55 & 10.24 & 10.48\\
   \hline\hline
    \multicolumn{1}{r}{\bf Global average} &  & {\bf  1.53} & {\bf 6.05} & {\bf 9.32} & {\bf 9.32} &{\bf  0.77} & {\bf 4.43} & {\bf 8.17} & {\bf 8.15}\\
   \hline
    \end{tabular}%
  \label{tab:TV_ISNR}%
\end{table*}%

\begin{table*}[htbp]
  \centering
  \caption{ISNR values (in {\rm dB}) obtained by the four algorithms tested (see the acronyms in the text)
  with the frame-based analysis formulation of image deblurring. }
    \begin{tabular}{lr|cccc|cccc|}
   	&		 & \multicolumn{4}{c|}{\bf Cameraman}  & \multicolumn{4}{c|}{\bf Lena} \\
    %\midrule
   \multicolumn{1}{c}{\bf Blur} & {\bf BSNR}     & {FA-BC} &  {FA-ET}   & {FA-CG}    & {FA-MD}
    & {FA-BC} &  {FA-ET}   & {FA-CG}    & {FA-MD}    \\
    \hline
    Uniform & 30dB   & -0.32 & 4.48  & 5.27  & 5.27  & -2.54 & 2.96  & 5.07  & 5.07  \\ % & -0.32 & 4.48  & 5.27  & 5.27  & -2.54 & 2.96  & 5.07  & 5.07 \\
    Out-of-focus & 30dB    & 0.78  & 5.17  & 5.53  & 5.53  & -1.50 & 4.33  & 5.51  & 5.50 \\
    Linear motion & 30dB    & 0.91  & 7.27  & 8.11  & 8.12  & -1.85 & 6.18  & 7.55  & 7.50 \\
    Gaussian & 30dB    & 1.11  & 2.98  & 2.93  & 2.92  & -0.32 & 2.90  & 2.62  & 2.63 \\
      \hline
    \multicolumn{1}{r}{Average for} & 30dB  & 0.62 & 4.98 & 5.46 & 5.46 & -1.55 & 4.09 & 5.19 & 5.18\\
   \hline\hline
    Uniform & 40dB      & -0.33 & 5.09  & 7.19  & 7.17  & -2.54 & 3.05  & 6.91  & 6.83 \\
    Out-of-focus & 40dB    & 0.78  & 7.33  & 8.51  & 8.51  & -1.50 & 4.88  & 7.95  & 7.95 \\
    Linear motion & 40dB    & 0.91  & 9.38  & 12.58 & 12.59 & -1.84 & 7.12  & 11.29 & 11.22 \\
    Gaussian & 40dB    & 1.10  & 3.79  & 3.79  & 3.80  & -0.33 & 3.56  & 3.48  & 3.48 \\
     \hline
   \multicolumn{1}{r}{Average for} & 40dB  & 0.62  & 6.40  & 8.02  & 8.02  & -1.55   & 4.65  & 7.41  & 7.37  \\
   \hline\hline
    Uniform & 50dB     & -0.34 & 6.06  & 10.00 & 9.99  & -2.53 & 3.06  & 8.95  & 9.02 \\
    Out-of-focus & 50dB    & 0.79  & 8.62  & 12.10 & 12.10 & -1.50 & 5.02  & 11.03 & 10.99 \\
    Linear motion & 50dB    & 0.91  & 9.86  & 16.89 & 16.82 & -1.84 & 7.29  & 14.70 & 14.63 \\
    Gaussian & 50dB    & 1.10  & 4.49  & 4.67  & 4.67  & -0.33 & 3.87  & 4.27  & 4.21 \\
 \hline
    \multicolumn{1}{r}{Average for} & 50dB  & 0.62 & 7.26  & 10.91  & 10.90  &  -1.55 & 4.81  & 9.74  & 9.71 \\
   \hline\hline
    Uniform & 60dB      & -0.33 & 6.59  & 12.32 & 12.52 & -2.52 & 3.06  & 10.21 & 10.63 \\
    Out-of-focus & 60dB      & 0.78  & 8.88  & 15.79 & 15.77 & -1.50 & 5.04  & 14.15 & 14.21 \\
    Linear motion & 60dB    & 0.92  & 9.92  & 20.37 & 20.34 & -1.84 & 7.31  & 16.31 & 16.41 \\
    Gaussian & 60dB      & 1.10  & 4.76  & 4.99  & 5.01  & -0.33 & 3.94  & 4.61  & 4.59 \\
\hline
    \multicolumn{1}{r}{Average for} & 60dB  & 0.62 & 7.54 & 13.37 & 13.41 & 0.77 & 4.84 & 11.32 & 11.46\\
   \hline\hline
    \multicolumn{1}{r}{\bf Global average} &  & {\bf  0.62} & {\bf 6.55} & {\bf 9.44} & {\bf 9.45} &{\bf  -1.55} & {\bf 4.60} & {\bf 8.41} & {\bf 8.43}\\
   \hline    \end{tabular}%
  \label{tab:WAV_ISNR}%
\end{table*}%

%
% \begin{table*}[htbp]
%  \centering
% \caption{Add caption}
%    \begin{tabular}{c|ccc|ccc|ccc|ccc|}
%            & \multicolumn{6}{c}{Cameraman} & \multicolumn{6}{c}{Lena} \\
%  BSNR & TV-MD & TV-CG & Ratio & FA-MD & FA-CG & Ratio & TV-MD & TV-CG & Ratio & FA-MD & FA-CG & Ratio \\
%  \hline
%    30    & 1.32  & 2.62  & 1.98  & 1.87  & 3.05  & 1.63  & 2.65  & 3.69  & 1.39  & 3.08  & 4.03  & 1.31 \\
%    40    & 1.52  & 2.96  & 1.94  & 2.39  & 3.65  & 1.53  & 2.85  & 4.13  & 1.45  & 3.60  & 5.04  & 1.40 \\
%    50    & 2.28  & 4.20  & 1.84  & 4.70  & 7.38  & 1.57  & 3.32  & 5.18  & 1.56  & 5.97  & 8.76  & 1.47 \\
%    60    & 8.63  & 16.54 & 1.92  & 16.25 & 26.87 & 1.65  & 11.02 & 21.10 & 1.91  & 16.28 & 25.77 & 1.58 \\
%  \hline
%   {\bf Average} & {\bf 3.44}  & {\bf 6.58}  & {\bf 1.91}  & {\bf 6.30}  & {\bf 10.24} & {1.62}  & {\bf 4.96}  & {\bf 8.53}  & {\bf 1.72}  & {\bf 7.23}  & {\bf 10.90} & {\bf  1.51} \\
%    \end{tabular}%
%  \label{tab:addlabel}%
%\end{table*}%

 \begin{table*}[htbp]
  \centering
 \caption{Computation times of the four algorithms tested; the times (in seconds) reported in
 each row correspond to averages over the four blurs at the indicated BSNR values.}
    \begin{tabular}{c|cc|cc|cc|cc|}
            & \multicolumn{4}{c|}{\bf Cameraman} & \multicolumn{4}{c|}{\bf Lena} \\
  {\bf BSNR} & TV-MD & TV-CG  & FA-MD & FA-CG  & TV-MD & TV-CG  & FA-MD & FA-CG  \\
  \hline
    30dB   & 1.32  & 2.62  &  1.87  & 3.05  &  2.65  & 3.69  & 3.08  & 4.03  \\
    40dB   & 1.52  & 2.96  &  2.39  & 3.65  &  2.85  & 4.13  & 3.60  & 5.04  \\
    50dB    & 2.28  & 4.20  &  4.70  & 7.38  &  3.32  & 5.18  & 5.97  & 8.76  \\
    60dB    & 8.63  & 16.54 &  16.25 & 26.87 &  11.02 & 21.10 & 16.28 & 25.77 \\
  \hline
   {\bf Average} & {\bf 3.44}  & {\bf 6.58}  & {\bf 6.30}  & {\bf 10.24} & {\bf 4.96}  & {\bf 8.53}   & {\bf 7.23}  & {\bf 10.90}  \\
   \hline
    \end{tabular}%
  \label{tab:total_time}
\end{table*}

Table~\ref{tab:total_time} shows the average times\footnote{The methods were implemented in MATLAB
and the experiments run on a Intel Core i5 processor at 2.39GHz.} taken by the four algorithms tested
at each of the BSNR values considered.
To obtain these results, the MD variant of each algorithm is first run using its own
convergence criterion (relative variation of the objective function falling below $10^{-4}$),
and then the CG variant is run until it reaches the same value of the
objective function. These results show that the MD versions are faster
than their CG-based counterparts by factors that range from (roughly) from 1.5 to 2.
This speed-up is confirmed by the average time per iteration of each of the four algorithms, shown in
Table~\ref{tab:time_iter}; each iteration of the MD variant is roughly 1.6  to 1.8 faster than the
CG counterpart.
Notice that, in addition to converging faster, the MD approach is also considerably simpler
to implement, since it does not involve the inner CG iterations. Since, as shown above,
both methods yield similar ISNR values, this allows concluding that, in the context of
ADMM-based deblurring algorithms, the mask decoupling approach is preferable to the
use of the Reeves technique \cite{Reeves05} with a CG-based inner step.

\begin{table}[bht]
  \centering
  \caption{Average time per iteration (in seconds) of each of the four algorithms considered.}
    \begin{tabular}{l|c|c|c|c}
        {\bf\small Algorithm} &  TV-MD    & TV-CG  & FA-MD    & FA-CG \\
\hline
{\bf\small Time per iteration} &  0.033 & 0.058 &   0.059 & 0.096  \\
    \end{tabular}%
  \label{tab:time_iter}%
\end{table}%

\begin{figure}[h]%[tb]
	\centering
	 \begin{tabular}{cc}
   \includegraphics[width=0.16\textheight]{x0_TV_camera256_motion_19x19__40dB.eps} &
   \includegraphics[width=0.16\textheight]{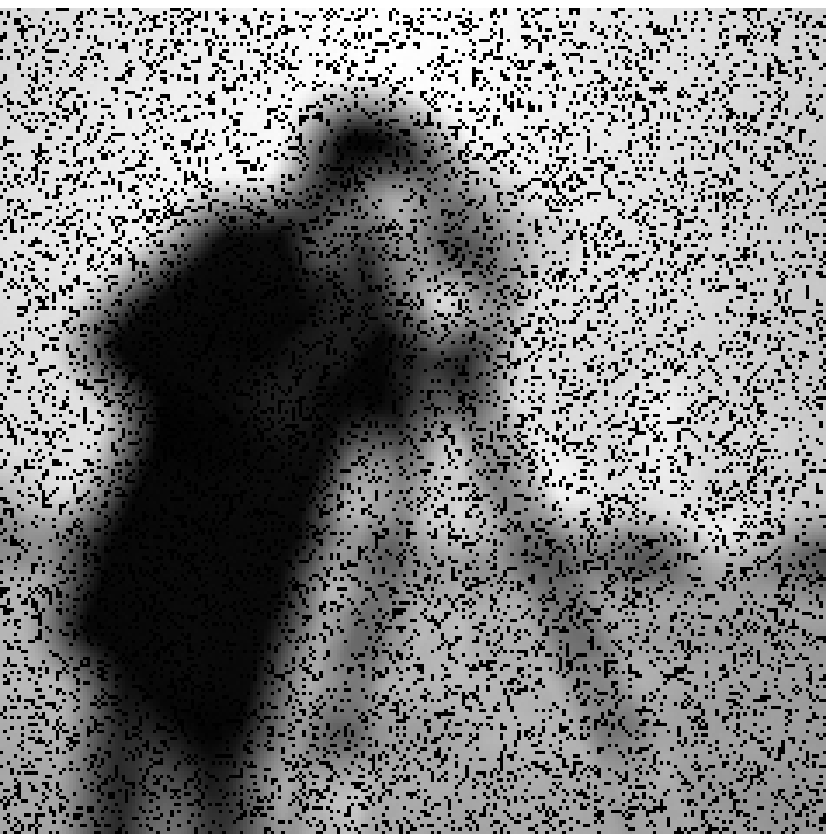} \\
   \vspace{0.15cm} {\small original  ($256\times 256$)}&  {\small observed ($238\times 238$)}\\
  \includegraphics[width=0.16\textheight]{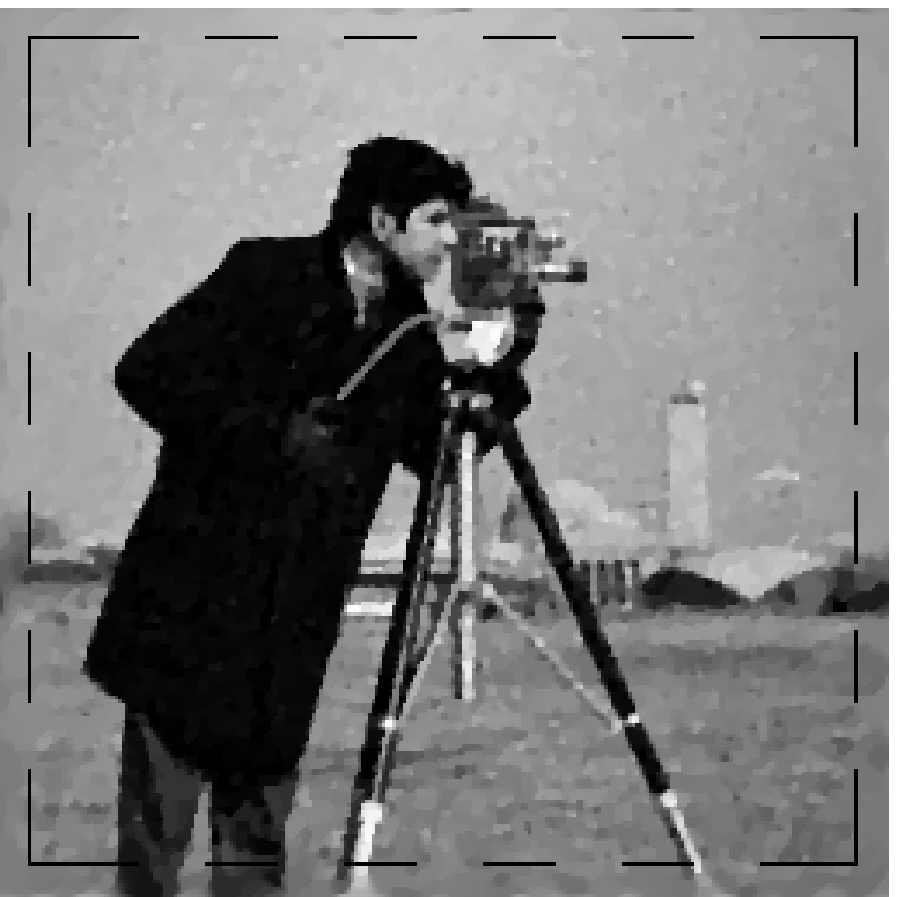}&
    \includegraphics[width=0.16\textheight]{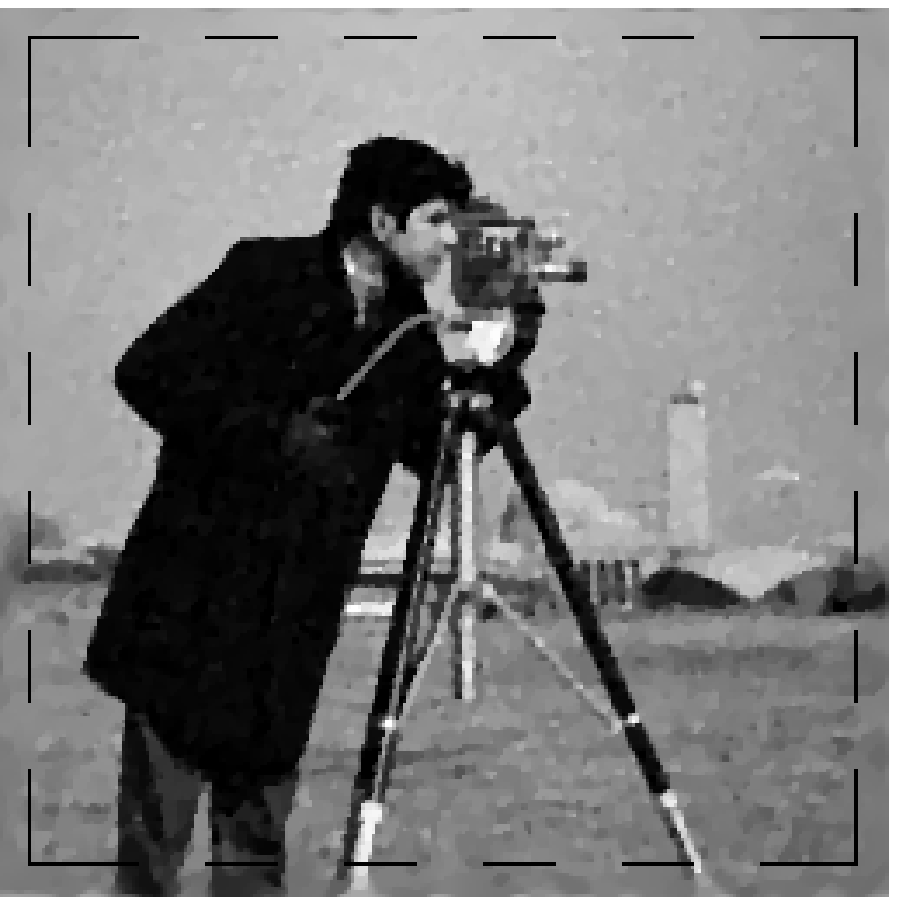} \\
   \vspace{0.25cm}  {\small FA-CG (SNR = 20.58dB)} & {\small  FA-MD (SNR = 20.57dB)}
   	 \end{tabular}
\caption{Illustration on the use of the proposed method for simultaneous deblurring and
inpainting. The observed image suffered a $19\times 19$ uniform blurring, followed by
the loss of 20\% of (randomly located) pixels. }
\label{fig:inpainting}
\end{figure}

%\subsection{Simultaneous deconvolution and inpainting}
%\label{ssec:Deb_Inp}
Finally, we illustrate the successful application of the proposed
TV-MD and TV-CG approaches in simultaneous non-periodic deblurring and inpainting (the FA-MD
and FA-CG algorithms yield very similar results).
In this case, the methods that assume periodic BC (with or
without the edgetaper) simply cannot be used. Figure~\ref{fig:inpainting}
shows the results obtained with the ``cameraman'' image, non-periodically
blurred  with a $19 \times 19$ uniform blur (using the same boundary
conditions as in the previous experiments), at 40dB BSNR, and
with $20\%$ missing pixels. Of course, these results are
not a full experimental assessment of the effectiveness of the
proposed algorithms for image inpainting, which will be the addressed
in future work.

%\begin{figure}[h]%[tb]
%	 \begin{tabular}{cc}
%   \includegraphics[width=0.17\textheight]{Y_Nblind_Camera_Unif9x9_40dB.eps} &
% \includegraphics[width=0.17\textheight]{Y_NblindInpaint_Camera_Unif9x9_40dB.eps} \\
%   (a) Non-cyclic blur &  (b) Non-cyclic blur \\
%   \ & with loss of pixels \\
%   \includegraphics[width=0.17\textheight]{Nblind_New_Camera_Unif9x9_40dB.eps} &  \includegraphics[width=0.17\textheight]{NblindInpaint_New_Camera_Unif9x9_40dB.eps} \\
%   (c) Result from (a)  &   (d) Result from (b)
%    \end{tabular}
%\caption{Results obtained, with our approach, for the cameraman image degraded with a non-cyclic convolution with a $9\times9$ uniform blur, at 40dB BSNR, followed by loss of $20\%$ of the pixels:  (a) blurred image, (b) blurred image with $20\%$
%missing pixels, (c) and (d) images recovered from (a) and (b), respectively. }
%\label{fig:Deb_Inp}
%\end{figure}
%

\section{Conclusions and Future Work}
\label{sec:Conc}

We have presented a new strategy to extend recent fast image deconvolution
algorithms, based on the {\it alternating direction method of multipliers} (ADMM),
to problems with unknown boundary conditions; previous versions of this
class of methods were limited to deconvolution problems with periodic
boundary conditions. In fact, the proposed algorithms are able to address a more general
class of problems, where the degradation model includes not only a convolution with
some blur filter, but also loss of pixels (the so-called image inpainting problem).
We have considered total-variation regularization as well as frame-based analysis and
synthesis formulations, and gave convergence guarantees for the algorithms proposed.
Experiments using large blur filters (where the effect of the unknown
boundaries is more evident) and several noise levels showed the adequacy
of the proposed approach.

Ongoing and future work includes the application and extension of the approach
herein developed to: video deblurring; image and video super-resolution; spatially
varying regularization. Another direction of current research concerns the application
of the proposed approach in iterative blind deconvolution \cite{Almeida_Figueiredo_BID_ADMM_2013}.

%\bibliographystyle{IEEEtranS}
%\bibliography{Ref_NcircADMM}

% Generated by IEEEtranS.bst, version: 1.13 (2008/09/30)

\end{document}